\documentclass[11pt,reqno]{amsart}
\pdfoutput=1
\usepackage{amssymb}
\usepackage{graphicx}
\usepackage{enumerate}

\catcode`\@=11

\long\def\@savemarbox#1#2{\global\setbox#1\vtop{\hsize\marginparwidth 
  \@parboxrestore\tiny\raggedright #2}}
\newcommand\lref[1]{\ref{#1}%
\@ifundefined{r@DisplaY #1}{}{ (#1)}}

\newcommand\fakelabel[2]{\@bsphack\if@filesw {\let\thepage\relax
   \newcommand\protect{\noexpand\noexpand\noexpand}%
\xdef\@gtempa{\write\@auxout{\string
      \newlabel{#1}{{#2}{\thepage}}}}}\@gtempa
   \if@nobreak \ifvmode\nobreak\fi\fi\fi\@esphack}

\def\SL@margintext#1{{\showlabelsetlabel{\tiny\{\SL@prlabelname{#1}\}}}}
\catcode`\@=12

\def\Empty{}
\newcommand\oplabel[1]{
  \def\OpArg{#1} \ifx \OpArg\Empty {} \else
        \label{#1}
  \fi}
\newtheorem{theoremSt}{Theorem}[section]

\newtheorem{exampleSt}[theoremSt]{Example}
\newtheorem{exerciseSt}[theoremSt]{Exercise}

\newcommand\MakeStEnv[1]{
  \newenvironment{#1}[1]{
  \begin{#1St} \oplabel{##1}%
  \global\def\CrntSt{\thetheoremSt}%
}{ 
  \end{#1St} }
  \newenvironment{#1+}[1]{
  \begin{#1St} \label{##1}%
  \label{DisplaY ##1}%
  \global\def\CrntSt{\thetheoremSt}%
  \def\Labl{##1}\ifx\Labl\Empty{} \else {\em (\Labl)\,}\fi%
}{ 
  \end{#1St} }
}
\MakeStEnv{theorem}
\MakeStEnv{corollary}
\MakeStEnv{proposition}
\MakeStEnv{lemma}
\MakeStEnv{definition}
\MakeStEnv{conjecture}
\MakeStEnv{problem}
\MakeStEnv{question}

\newcommand\restate[3]{
\medskip\par\noindent
{\bf #1 \ref{#2}} 
{\it #3}
\par\medskip
}

\newlength{\saveu}

\newenvironment{pf*}[1]{%
 \begin{proof}[#1]%
}{ 
 \end{proof}
}

\newcommand{\finishproof}[1]{ 
  \def\FPArg{#1}
  \ifx\FPArg\Empty
        \newcommand\FPArg{\CrntSt}  \fi
  \smallbreak\noindent\makebox[\textwidth]{\hfill\fbox{\FPArg}}
  \medbreak\noindent
}

\newcommand\AAA{{\mathcal A}}

\newcommand\CC{{\mathcal C}}

\newcommand\EE{{\mathcal E}}
\newcommand\FF{{\mathcal F}}

\newcommand\HH{{\mathcal H}}

\newcommand\LL{{\mathcal L}}
\newcommand\MM{{\mathcal M}}
\newcommand\NN{{\mathcal N}}

\newcommand\PP{{\mathcal P}}

\newcommand\UU{{\mathcal U}}

\newcommand\CH{{\CC\HH}}

\newcommand\PMF{{\PP\kern-2pt\MM\FF}}

\newcommand\PML{{\PP\kern-2pt\MM\LL}}

\newcommand\Area{\operatorname{Area}}

\newcommand\ep{\epsilon}

\newcommand\hhat{\widehat}

\newcommand\union{\cup}
\newcommand\intersect{\cap}
\newcommand\bbR{{\mathord{\text{I\kern-2pt R}}}}        
\newcommand\bbH{{\mathord{\text{I\kern-2pt H}}}}        

\newcommand\C{{\mathbb C}}

\newcommand\Z{{\mathbb Z}}
\newcommand\R{{\mathbb R}}

\newcommand\Hyp{{\mathbb H}}

\newcommand\PSL[1]{\text{PSL}_{#1}}

\newcommand\bigrightarrow[1]{\hbox to #1{\rightarrowfill}}
\newcommand\bigleftarrow[1]{\hbox to #1{\leftarrowfill}}

\newcommand\boundary{\partial}
\newcommand\semidir{\mathrel{\hbox{\vrule depth-.03ex height1.1ex\kern-0.15em$\times$}}}

\newcommand\del{\nabla}
\newcommand\til{\widetilde}

\newcommand{\ssm}{\setminus}
\newcommand{\diam}{\operatorname{diam}}

\renewcommand{\Im}{\operatorname{Im}}

\numberwithin{equation}{section}

\catcode`\@=11

\def\subsection{\@startsection{subsection}{2}%
  \z@{.5\linespacing\@plus.7\linespacing}{.5em}%
  {\normalfont\bfseries\centering}}

\def\section{\@startsection{section}{1}%
  \z@{.7\linespacing\@plus\linespacing}{.5\linespacing}%
  {\normalfont\large\bfseries\centering}}

\def\subsubsection{\@startsection{subsubsection}{3}%
  \z@{.5\linespacing\@plus.7\linespacing}{-.5em}%
  {\normalfont\bfseries}}

\catcode`\@=12

\newcommand{\inj}{\operatorname{inj}}

\newcommand{\collar}{\operatorname{\mathbf{collar}}}

\newcommand{\topprec}{\prec_{\rm top}}

\newcommand{\fsubd}{\mathrel{{\scriptstyle\searrow}\kern-1ex^d\kern0.5ex}}
\newcommand{\bsubd}{\mathrel{{\scriptstyle\swarrow}\kern-1.6ex^d\kern0.8ex}}
\newcommand{\fsubeq}{\mathrel{\raise-.7ex\hbox{$\overset{\searrow}{=}$}}}
\newcommand{\bsubeq}{\mathrel{\raise-.7ex\hbox{$\overset{\swarrow}{=}$}}}

\newcommand{\base}{\operatorname{base}}

\newcommand{\bbar}{\overline}

\newcommand{\EL}{\mathcal{EL}}

\newcommand{\tsh}[1]{\left\{\kern-.9ex\left\{#1\right\}\kern-.9ex\right\}}

\newcommand\MT{{\mathbb T}}
\newcommand\Teich{{\mathcal T}}

\newcommand{\id}{{\operatorname{id}}}
\newcommand{\cE}{{\mathcal E}}

\begin{document}

\title{Convergence properties of end invariants}

\author[J. Brock]{Jeffrey F. Brock}
\address{Brown University}
\author[K. Bromberg]{Kenneth W. Bromberg}
\address{University of Utah}
\author[R. Canary]{Richard D. Canary}
\address{University of Michigan}
\author[Y. Minsky]{Yair N. Minsky}
\address{Yale University}
\date{\today}
\thanks{The authors gratefully acknowledge support from the National
  Science Foundation}

\begin{abstract}
We prove a continuity property for ending invariants of convergent
sequences of Kleinian surface groups. We also analyze the bounded curve
sets of such groups and show that their projections to non-annular subsurfaces
lie a bounded Hausdorff distance from geodesics joining the projections
of the ending invariants.
\end{abstract}

\maketitle



\newcommand\epzero{\ep_0}
\newcommand\epone{\ep_1}
\newcommand\epotal{\ep_{\rm u}}
\newcommand\kotal{k_{\rm u}}
\newcommand\Kmodel{K_0}
\newcommand\Kone{K_1}
\newcommand\Ktwo{K_2}
\newcommand\bdry{\partial} 
\newcommand\stab{\operatorname{stab}}
\newcommand\nslices[2]{#2|_{#1}}
\newcommand\ME{M\kern-4pt E}
\newcommand\bME{\overline{M\kern-4pt E}}
\renewcommand\del{\partial}
\newcommand\s{{\mathbf s}}
\newcommand\pp{{\mathbf p}}
\newcommand\qq{{\mathbf q}}
\newcommand\uu{{\mathbf u}}
\newcommand\vv{{\mathbf v}}
\newcommand\zero{{\mathbf 0}}
\newcommand{\cB}{{\mathcal B}}
\newcommand\mm{\operatorname{\mathbf m}}
\newcommand{\dehntw}{\theta}  
\newcommand{\hull}{\operatorname{hull}}
\renewcommand{\CH}{{\mathcal C}\kern-2pt\mathit{H}}
\newcommand{\thd}{{\mathbf d}|}
\newcommand{\thdiam}{\mathbf{diam}|}
\newcommand{\core}{{\mathcal K}}

\section{Introduction}

The solution \cite{ELC1,ELC2} of Thurston's Ending Lamination
Conjecture (together with that of Marden's Tameness Conjecture
\cite{agol,calegari-gabai}) gives a complete classification of
finitely-generated Kleinian groups in terms of their topological type
and their {\em end invariants}. This classification leaves an
incomplete picture, however, because it does not describe the {\em
  topology} of the deformation space of hyperbolic structures
associated to a given group (with the natural topology induced from
representation spaces). In particular, the end invariant data does not
vary continuously with deformations in any of the usual topologies
that have arisen historically \cite{brock-invariants,AC-pages}. 
Moreover, such deformation spaces can fail to be locally connected \cite{bromberg-PT,magid}.
In this article, we describe how end invariants do converge in limiting
families of hyperbolic structures. In the process, we produce a number of important
structural refinements to the geometric picture developed in
\cite{ELC1,ELC2}.

We restrict ourselves to {\em Kleinian surface groups}, which are
discrete, faithful representations $\rho:\pi_1(S)\to \PSL 2(\C)$ where
$S$ is an oriented compact surface (a parabolicity condition is
imposed on $\boundary S$ if it is nonempty).  Let $AH(S)$ denote the
space of conjugacy classes of such representations, viewed as a subset
of the $\PSL 2(\C)$ character variety of $\pi_1(S)$.  The end
invariants of $[\rho]\in AH(S)$ are a pair of data $\nu^\pm(\rho)$,
each a union of marked Riemann surface structures and geodesic laminations
supported on essential subsurfaces of $S$ (see \S\ref{back} for
details). The orientation of $S$ and of the quotient manifold $N_\rho
= \Hyp^3/\Im(\rho)$ give $N_\rho$ a ``top'' and ``bottom'' side or
{\em end}, with asymptotic geometry encoded by $\nu^+$ and $\nu^-$,
respectively.

\subsubsection*{Limits of projections of end invariants}
The primary objective of \cite{ELC1,ELC2}, as well as their precursors
\cite{minskyPT,minsky:kgcc}, is to obtain coarse information about
$N_\rho$ using the projections of $\nu^+$ and $\nu^-$ to the
 {\em curve complexes} $\CC(W)$ where \hbox{$W\subseteq S$} denotes an
essential subsurface of $S$. Let $\pi_W(\nu^\pm)$ denote these
projections. (We emphasize that we allow the possibility that $W=S$.) 
We recall that $\CC(W)$ is a $\delta$-hyperbolic metric
space \cite{masur-minsky}, and that (for $W$ nonannular) its Gromov
boundary can be identified with $\EL(W)$, the space of {\em unmeasured
filling laminations} in $W$ \cite{klarreich}. Moreover $\EL(W)$ is also
the set of laminations that can occur as components of the end
invariants $\nu^\pm$ supported on non-annular $W$. Our first theorem
describes a sense in which the end invariants in a convergent
sequence of representations can be said to converge, establishing a
continuity property for the projections of end invariants to
subsurfaces.

\begin{theorem}{limits}
Let $\rho_n\to\rho$ in $AH(S)$.
If $W \subseteq S$ is an essential subsurface of $S$, other than an annulus or a pair of pants,
and $\lambda\in\EL(W)$ is a lamination supported on $W$, the following statements are
equivalent:
\begin{enumerate}
\item
$\lambda$ is a component of $\nu^+(\rho)$.
\item
$\{\pi_W(\nu^+(\rho_n))\}$ converges to $\lambda$
\end{enumerate}
Furthermore we have,
\begin{enumerate}[\indent (a)]
\item
if $\{\pi_W(\nu^+(\rho_n))\}$ accumulates on $\lambda \in \EL(W)$ then it
  converges to $\lambda$,
\item
the sequences $\{\nu^+(\rho_n)\}$ and $\{\nu^-(\rho_n)\}$ do not converge to a common
$\lambda \in \EL(S)$, and
\item
if $W \subsetneq S$ is a proper subsurface then convergence of
  $\{\pi_W(\nu^+(\rho_n))\}$ to $\lambda \in \EL(W)$ implies
  $\{\pi_W(\nu^-(\rho_n))\}$ does not accumulate on $\EL(W)$.
\end{enumerate} 
The same statements hold with ``$+$'' replaced by  ``$-$''.

\end{theorem}

We remark that Ohshika has obtained a similar result in \cite[Theorem
2]{ohshika-divergence}, phrased in the equivalent language of
Hausdorff limits. One can make the hybrid objects $\nu^\pm(\rho_n)$ into laminations by replacing each Riemann surface component of $\nu^\pm(\rho_n)$ with bounded length pants decompositions on the associated hyperbolic metric. We then let $\lambda^\pm$ denote the Hausdorff limit of these sequences.
The statement that 
$\pi_W(\nu^+(\rho_n))$ converges to $\lambda\in\EL(W)$ is then equivalent
to the condition that
$\lambda^+$ contains $\lambda$ as a component. 

This convergence behavior was presaged in the examples of
\cite{brock-invariants} for representations in a {\em Bers slice}, and
those examples also indicate 
how the Hausdorff topology on end invariants {\em
  necessarily} fails to predict the full end-invariant of the limit.
In particular, one might hope to find that the parabolic components of
$\nu^\pm(\rho)$  always arise either as components of the Hausdorff
limits $\lambda^+$ or $\lambda^-$, or as
boundary components of 
subsurfaces $W$ filled by components $\lambda$ of these Hausdorff
limits. However, in
\cite{brock-invariants} examples are given in which parabolic curves
in the limiting invariants are not related to the Hausdorff limit in
either of these ways.

In the other direction, the phenomenon of {\em wrapping} explored in
detail in \cite{AC-pages} gives examples in which {\em both} Hausdorff limits $\lambda^+$ and $\lambda^-$
contain the same curve as a component, but the curve can only appear
in the end invariant of one side or the other in the limit.

We do not address these subtleties here, but note that all parabolics
and ending laminations are predicted in full by recording the full
collection of subsurfaces for which such projections diverge, or
equivalently, by studying the sequence of {\em hierarchies} associated
to the end invariants. The precise behavior and its connection to end
invariants will be described in \cite{BBCL}.

Theorem~\ref{limits}, together with Theorem~\ref{bounded} below, is
used in our related paper \cite{nobumping} to analyze (and rule out)
``bumping'' phenomena on the boundary of $AH(S)$, and in particular to
identify boundary points where $AH(S)$ is locally connected.
Theorem~\ref{limits} will also be applied, together with Theorem
\ref{topistop}, in \cite{BBCL} which
gives a complete characterization, in terms of end invariants, of
convergence and divergence of sequences of Kleinian surface groups.

\medskip

\subsubsection*{Controlling the bounded curve sets}
The second theme of this paper involves improving our understanding of
the {\em bounded curve sets} associated to a Kleinian surface group.
In \cite{ELC1,ELC2} we applied the notion of a {\em hierarchy of
  geodesics} as developed in Masur-Minsky \cite{masur-minsky2}. This
combinatorial device connects the two end invariants with a family of
{\em markings}, curve systems on $S$, in a combinatorially efficient
way.  A crucial step in \cite{ELC1,ELC2} is to establish
a-priori bounds on the geodesic lengths of all simple closed curves that
appear in such a hierarchy.  On the other hand one can simply ask to
understand the full set
$$\CC(\rho,L) = \{ \alpha \in \CC(S) : \ell_\rho(\alpha) \le L \}$$
of simple closed curves in $S$ whose $\rho$-length is bounded by $L$
(for a given $L$). Our second theorem gives a description of this set
in terms of its {\em subsurface projections}.  We denote by
$\hull_W(\nu^+,\nu^-)$ the union of geodesics in $\CC(W)$ connecting
$\pi_W(\nu^+)$ to $\pi_W(\nu^-)$. (Hyperbolicity of $\CC(W)$ implies
that this union lies in a uniform neighborhood of any one of its
members). The set of curves appearing in the hierarchy has the
property that its projections into each $\CC(W)$ lie in
$\hull_W(\nu^\pm)$. The next theorem shows that the same holds for the
bounded curve set. Let $d_{\rm Haus}$ denote Hausdorff distance for
subsets of a metric space, applied below to $\CC(W)$.
We also use $d_W(x,y)$ as an abbreviation for $d_{\CC(W)}(\pi_W(x),\pi_W(y))$.

\begin{theorem}{bounded}{}
Given $S$, there exists $L_0$ such that for all $L\ge L_0$ there exists
$D=D(S,L)$, such that given $\rho\in AH(S)$ with end invariants $\nu^\pm$ and an essential
subsurface $W\subset S$  which is not an annulus or a pair of pants, 
$$
d_{\rm Haus}\left(\pi_W(\CC(\rho,L)) , \hull_W(\nu^\pm(\rho))\right) \le D.
$$
Moreover, if $d_W(\nu^+(\rho),\nu^-(\rho)) > D$ then 
$\CC(\rho,L) \intersect\CC(W)$ is nonempty and 
$$
d_{\rm Haus}\left(\CC(\rho,L)\intersect\CC(W) , \hull_W(\nu^\pm(\rho))\right) \le D.
$$
\end{theorem}

Our third theorem relates the projections of bounded-length curves to
their topological ordering in the manifold (in the sense described in
\S \ref{ordering}). 
It states that when the geodesic representative $\alpha^*$ of a
curve $\alpha\in\CC(\rho,L)$ lies above the geodesic representative $\beta^*$ of some component  $\beta$ of
the boundary of a subsurface $W$ which it overlaps, then its projection to $\CC(W)$ is
uniformly close to $\pi_W(\nu^+)$. (We recall that $\alpha^*$ lies above $\beta^*$ if $\alpha^*$ can be pushed arbitrarily
far upward, in the complement of $\beta^*$,  in the product structure on $N_\rho\cong S\times \R$.)
This property follows directly from
the machinery of \cite{ELC2} in the case of curves that arise in the
hierarchy of $N_\rho$ (see Lemma \ref{top order and geodesic}).

\begin{theorem}{topistop}
Given $S$ and $L>0$ there exists $c$ such that, given
$\rho\in AH(S)$, an essential subsurface
$W\subset S$ which is not  a pair of pants, 
and a curve $\alpha\in\CC(\rho,L)$ such that
$\alpha^*$ lies above the geodesic representative of some
component of $\partial W$ that it overlaps, then
$$d_W(\alpha,\nu^+(\rho))\le c.$$
Furthermore, if $W$ is not an annulus
or a pair of pants, $\alpha\in\CC(\rho,L)$ overlaps $\partial W$, and
$$ d_W(\alpha,\nu^-) > c$$
then $\alpha^*$ lies above the geodesic representative of every component of $\partial W$
that it overlaps.

The same holds when replacing ``above'' with ``below'' and $\nu^+$
with $\nu^-$.
\end{theorem}

We note that the conclusion of Theorem \ref{bounded} need not hold
in the case that $W$ is an annulus. It is possible that $N_\rho$ contains a 
bounded geometry pleated surface which is ``wrapped'' several, say $n$, times about
the Margulis tube $\MT(\beta)$ associated to the core curve of $W$. If $\alpha$ is a
curve on the pleated surface of bounded length, say $L$, which overlaps $\beta$, then $\alpha$ may be 
concatenated with $n$ copies of the meridian of the Margulis tube of $\beta$ 
to obtain a curve $\alpha'$ of length roughly $L+nC$ whose geodesic representative 
lies above or below $\beta^*$. Moreover, $d_W(\alpha,\alpha')$ is roughly $n d_W(\nu^+,\nu^-)$.
For any given value of $n$, one may construct families of examples where $d_W(\nu^+,\nu^-)$ is arbitrarily large, but
one may make uniform choices of $L$ and $C$. 
(This wrapping construction was introduced in \cite {AC-pages}, see also \cite[Lemma A.4]{mcmullen-ce} or
\cite{canary-bump}.)

\subsubsection*{Outline of the paper}
In Section \ref{back} we review background on curve complexes,
hierarchies, Kleinian surface groups and their end invariants. We also
review some material from our previous work in \cite{ELC1,ELC2},
particularly the structure of model manifolds associated to
hierarchies, and some consequences: In \ref{ordering}, and
particularly Lemma \ref{top order and geodesic}, we discuss the
relationship between combinatorial order relations in a hierarchy, and
its connection to a topological ordering in the corresponding
3-manifold.  In subsection \ref{thick product} we discuss $W$-product
regions, which are submanifolds of either the model manifold or the
hyperbolic manifold which are homeomorphic to $W\times [0,1]$ (for
some subsurface $W$) and so that $\partial W\times [0,1]$ is
identified with a submanifold of the boundaries of the tubes
associated to $\partial W$.  Lemma \ref{W cross sections} provides
criteria on a hierarchy that imply the existence of ``large''
$W$-product regions in the associated 3-manifolds.

In Section \ref{curve systems} we study the question of which
curves from a hierarchy are ``visible'' in a 
pleated surface (or any Lipschitz surface)
in a Kleinian surface group. Lemma \ref{hierarchy system} provides a
bounded-length system of hierarchy curves in every such surface,
satisfying some additional bounded-projection properties. This lemma
plays a central role in each of the main theorems. 

In Section \ref{bounded curve set} we prove Theorem \ref{bounded}. The
main new ingredient here is provided by Lemma \ref{hierarchy system}. 

In Section \ref{endlams} we prove Theorem \ref{limits}. We remark that
the principal difficulty in the proof involves showing that a
component of the limiting lamination corresponding to the top
invariants of a sequence is in fact a top invariant for the limit, and
not a bottom invariant in the limit.  The issue of such possible
``flipped ends'' in the limit has a long history in this subject,
arising first in the work of Thurston on strong limits of
quasifuchsian groups.  It arises in our proof in \cite{ELC2} of the
bilipschitz model theorem as well, and the relevant arguments there
contain echos of Thurston's original interpolation argument. In the
present paper, we rely primarily on properties of the bilipschitz
model, with Lemma \ref{top order and geodesic} on topological ordering
and Lemma \ref{W cross sections} on the existence of thick product
regions playing a central role.

In Section \ref{overlap} we give the proof of Theorem \ref{topistop}. 
We remark that the conclusion of the theorem is already known
from the properties of hierarchies and models, when the curves in
question are hierarchy curves. Thus we must answer the
question of how close the given curves of bounded length are to being
hierarchy curves, and Lemma \ref{hierarchy system} provides the needed
connection via pleated surfaces. The product region lemma \ref{W cross
  sections} then gives the necessary control of these pleated
surfaces. This argument, for the case of a non-annular surface, is
detailed in \S\ref{overlap nonannular}, whereas for the case of
annuli a fairly different argument is needed, which appears in
\S\ref{overlap annular}.

\section{Background}
\label{back}

\subsection{Curve complexes and laminations}
\label{complexes}

We briefly recall definitions and terminology from \cite{ELC2},
\cite{masur-minsky2} and related papers. We will denote by $\CC(S)$
the curve complex of a surface $S$ of finite type, recalling that it
is a locally infinite complex which is $\delta$-hyperbolic with
respect to a natural path metric \cite{masur-minsky}. Vertices of
$\CC(S)$ are isotopy classes of essential closed curves in $S$, and
simplices correspond to systems of disjoint curves (with a few
standard exceptions). The {\em curve and arc complex} $\AAA(S)$ is
formed similarly, with vertices corresponding to essential properly embedded
arcs (up to isotopy rel boundary)  as well as curves.

Klarreich's theorem \cite{klarreich} states that the Gromov boundary
$\boundary\CC(S)$ is naturally identified with $\EL(S)$, the set of
{\em filling geodesic laminations} in $S$, with topology inherited
from the space of measured laminations. (See \cite{} for background on
Thurston's measured lamination space). 

\subsubsection*{Markings}
A {\em marking} on $S$, in the sense of \cite{masur-minsky2}, is a
system of curves (i.e. a simplex of $\CC(S)$) together with a selection
of {\em transversal curves}, at most one for each curve in the
system. Each transversal  intersects the curve it is associated with
at most two times, and is disjoint from the others. 
The simplex of a marking $\mu$ is denoted $\base(\mu)$. If the
base is a pants decomposition of $S$ and every curve has a transversal
we call the marking {\em complete}. 

A {\em generalized marking} on $S$ is a similar object, except that
$\base(\mu)$ is allowed to have components that are minimal geodesic
laminations, not just simple closed curves. 

\subsubsection*{Subsurface projections}
Given an essential non-annular subsurface $W\subset S$,
there is a natural map $\pi_{\AAA(W)}:
\CC(S)\to \AAA(W)\union\{\emptyset\}$, which assigns to a curve system
in $S$ the barycenter of the span of the components of its essential intersection with $W$ (or
$\emptyset$ if there are none). 

A natural construction takes vertices of $\AAA(W)$ to points in
$\CC(W)$: Given a proper arc (or curve) $a\subset W$ take the
essential components of a 
regular neighborhood of $a\union\boundary W$. Composing this with
$\pi_{\AAA(W)}$ we obtain a map $\pi_W$ which takes vertices of
$\CC(S)$ to (finite sets of) vertices in $\CC(W)$.

For a marking $\mu$ in $S$ we can define $\pi_W(\mu)\subset\CC(W)$ as
the union of $\pi_W(\beta)$ over the curves $\beta$ in $\mu$. The union
has uniformly bounded diameter. For a generalized marking we need to allow $\pi_W$ to take values in $\CC(W) \cup \EL(W)$. If $\mu$ contains a minimal component $\lambda \in \EL(W)$ then $\pi_W(\mu) = \lambda$. If not than as above $\pi_W(\mu)$ is the union of $\pi_W(\beta)$ over the closed curves $\beta$ in $\mu$.

Complexes and projections can be defined for annuli also, with some
care: If $A$ is an annulus and $\gamma$ its core curve, we consider
the annular lift of $S$ associated to $A$, which has a natural
compactification coming from the circle at infinity of $\til
S$. Vertices of $\AAA(S)$ are essential arcs in this annulus, up to
homotopy fixing endpoints.  Given a
curve $\alpha$ in $S$ that crosses an annulus $A$ essentially, lift
$\alpha$ to the annular cover and keep only those components that
cross the annulus (or select one arbitrarily) to obtain
$\pi_A(\alpha)$.

Given generalized markings (or curves) $\alpha$ and $\beta$ which intersect $W$
essentially,  we regularly use the shorthand
$d_W(\alpha,\beta)$ to denote $d_{\CC(W)}(\pi_W(\alpha),\pi_W(\beta))$.
If $\gamma$ is the core of an annulus $A$, we write $d_A$ and
$d_\gamma$ interchangeably.

\subsection{Kleinian surface groups and end invariants}
\label{surface groups}

Let $AH(S)$ denote the space of Kleinian surface groups, i.e. discrete
faithful representations $\rho:\pi_1(S)\to \PSL 2(\C)$ taking
peripheral elements to parabolics, and considered up to conjugacy in
the image. The {\em end invariants} of $\rho\in AH(S)$ are two
hybrid objects $\nu^\pm(\rho)$, each a combination of laminations and
conformal structures on subsurfaces of $S$. We sketch a description
here, referring to \cite{ELC1,ELC2} and the references therein for more details. 

Let $N=N_\rho = \Hyp^3/\rho(\pi_1(S))$ be the quotient 3-manifold, and
let $N^0$ denote $N$ minus the (open) cusp neighborhoods associated to
the parabolic subgroups of $\rho(\pi_1(S))$ (which we note include one
cusp for each component of $\boundary S$). This manifold with boundary
has a {\em relative compact core} $\core\subset N^0$ which meets each cusp
boundary in one core annulus. Thus $\core$ can be identifed with
$S\times[-1,1]$, and $\core\intersect\boundary N^0$ is a union $P$ of
annuli in $\boundary \core$ that includes $\boundary S \times
[-1,1]$. 
We further decompose $P$ into the union $P^+$ of components of $P$
in $\boundary S\times \{1\}$ and the remaining components $P^-$.  This
decomposition has the property that no two annuli in $P$ are pairwise
isotopic in $S \times I$.

The closure of each component $W$ of $\boundary \core \setminus P$
bounds a component $U_W$ of $N^0\setminus \core$, which is a {\em
  neighborhood of an end of $N^0$.} We say that $W$ {\em faces} this
end and vice versa, and there are two possibilities for its geometry:
\begin{itemize}
\item {\em   geometrically finite}: it corresponds to a component
of the boundary at infinity of $N_\rho$, and $W$ inherits a finite-type
conformal structure, i.e. a point in $Teich(W)$. The convex core of
$N$ intersects $U_W$ in a bounded set. 

\item
{\em simply degenerate}: it is described by an {\em ending
  lamination}, which is a filling geodesic lamination in $W$, i.e. an
element of $\EL(W)$. This lamination is the support of the limit (in
Thurston's projective lamination space) of any sequence of curves in
$W$ whose geodesic representatives exit every bounded subset of
$U_W$. 
\end{itemize}
The end invariant $\nu^+(\rho)$ is a list of the following data: The
core curves of the annuli of $P^+$ that lie in $S\times\{+1\}$; the
conformal structures associated to geometrically finite ends facing 
subsurfaces in $S\times \{+1\}\setminus P^+$, and the laminations
associated to simply degenerate 
ends facing subsurfaces in $S\times \{+1\}\setminus P^+$.  The
invariant $\nu^-(\rho)$ is defined similarly; the ends associated to
$\nu^+$ and to $\nu^-$ are called upward-pointing and
downward-pointing, respectively.

\medskip

We recall here Thurston's notion of 
a {\em pleated surface} (or map), which 
is a map $f:X\to N$ where $X$ is a hyperbolic surface and $N$ a
hyperbolic 3-manifold, such that $f$ is length-preserving and totally
geodesic on the strata of a geodesic lamination on $X$. In the setting
of a Kleinian surface group $\rho:\pi_1(S)\to \PSL 2(\C)$, we
typically consider pleated maps with underlying surface $S$, in the
homotopy class determined by $\rho$. We say that such a map {\em
  realizes} a lamination $\lambda$ if it maps the leaves of $\lambda$
geodesically. 

The laminations and parabolic components of the end invariants are
exactly those laminations which are {\em unrealizable} in $\rho$.
So for example if $\nu^+(\rho)$ is a single lamination that fills $S$,
there is no pleated map that carries $\nu^+$ geodesically, and
moreover if $\gamma_n$ is a sequence of closed curves converging to
$\nu^+$ then a sequence of pleated surfaces realizing $\gamma_n$ will
necessarily escape every compact subset of $N_\rho$ and converge to
the end associated to $\nu^+$.

\subsubsection*{End markings}
In order to have a more topological object to work with, in
\cite[Section 7.1]{ELC1} we convert
the end invariants $\nu^\pm$ to a pair of generalized markings
$\mu^\pm$, as follows: for each conformal structure on a subsurface
$W$ we select a minimal-length complete marking on $W$. The union of
these with core curves of the annuli $P^+$ and the
lamination components of $\nu^+$ will be the generalized marking $\mu^+$; define
$\mu^-$ similarly.  Note that the total length of $\base(\mu^\pm)$ is
bounded by the {\em Bers Constant,} $L_B$, which bounds the length of the
minimal curve system in any hyperbolic structure on $S$ \cite{bers-constant}.

With this in mind we can define the projections $\pi_Y(\nu^\pm)=\pi_Y(\mu^\pm)$ for
any essential non-annular subsurface $Y$ which is not  a core curve of a component of the parabolic locus
$P^\pm$.

There is a bit of flexibility in this definition, as the choice of
markings in the geometrically finite subsurfaces may not be
unique. For our purposes this will not matter, as the different
choices have $\pi_Y$ images differing by a uniform amount, 
and moreover convergence conditions of the type $\pi_Y(\nu^+_n)\to
\lambda\in \EL(Y)$ are unaffected by the choices in the definition.

We also record a consequence of these definitions and the basic
properties of pleated surfaces: for an essential subsurface $W$, we have
\begin{equation}\label{mu contained 1}
\pi_W(\mu^\pm) \cap \overline{\pi_W(\CC(\rho,L))}\ne\emptyset
\end{equation}
provided $L$ is at least the Bers constant $L_B$. 
(Recall from the introduction that $\CC(\rho,L)$ is the set of
essential simple closed curves in $S$ whose $\rho$-length is bounded
by $L$).  If $W$ intersects a closed curve component $\beta$ of $\base(\mu)$, then
$l_\rho(\beta)\le L_B$ and so $\pi_W(\beta)\in\pi_W(\CC(\rho,L))$.
If $\lambda$ is a lamination component of $\base(\mu^\pm)$ and $Z={\rm supp}(\mu)$,
then there exists a family of pleated surfaces $\{f_n:X_n\to N_\rho\}$ with base surface
$Z$ which exit the end associated to $\lambda$ (see Bonahon \cite{bonahon}). If we choose
shortest curves $\beta_n$ on $Z_n$, then $l_\rho(\beta_n)\le L_B$, so
$\{\beta_n\}\subset \CC(\rho,L)$ and $\beta_n\to\lambda$.
If $W$ overlaps $Z$, then $\pi_W(\beta_n)\to \pi_W(\lambda)$, so
$$\pi_W(\lambda)\in \pi_W(\mu^\pm) \cap \overline{\pi_W(\CC(\rho,L))}.$$

\subsubsection*{Margulis tubes}

We fix throughout a Margulis constant $\ep_1$ for $\Hyp^3$, which it will be convenient
to take to be the same choice of Margulis constant as in \cite{ELC1} (see page 19) and \cite{ELC2}. 
In particular, this number is sufficiently small that the $\ep_1$-thin part of a
hyperbolic 3-manifold is a disjoint union of cusps and solid-torus
neighborhoods of geodesics.

If $\alpha$ is a curve in $S$ and $\rho$ is a given Kleinian surface group
we let $\MT(\alpha)$ denote the component of the $\ep_1$-thin part $(N_\rho)_{thin(\ep_1)}$
whose core is in the homotopy class of $\alpha$.  If $\ep<\ep_1$, then we define
$$\MT_\ep(\alpha)=\MT(\alpha)\cap (N_\rho)_{thin(\ep)}.$$

\subsection{Hierarchies}
\label{hierarchies}

Given two generalized markings $\mu^+$ and $\mu^-$, we let
$H(\mu^+,\mu^-)$ or $H(\mu^\pm)$ denote the hierarchy connecting them,
in the sense of \cite{masur-minsky2}, \cite{ELC1}  and \cite{ELC2}. 
We also denote this by $H(\nu^\pm)$, if $\mu^\pm$ are obtained from a
pair of end invariants $\nu^\pm$.
We give an
impressionist discussion here, referring the reader to
those three articles for the details. 
A hierarchy is a collection of {\em tight geodesics} supported on
subsurfaces of $S$, and interlocked in a structure that encodes certain
nesting and ordering properties. Each tight geodesic is essentially a
directed geodesic in the curve complex of the subsurface it is supported on. 
We typically denote such a geodesic $k_W$ if $W$ is the supporting
surface, and we let $\iota_W$  and $\tau_W$ denote the initial and
terminal vertices. 

We will use $\CH(\mu^\pm)$, or sometimes $\CH$, to denote
the set of all vertices of $\CC(S)$
which occur in the (non-annular) geodesics in a hierarchy $H(\mu^\pm)$.

A {\em resolution} of $H(\mu^\pm)$ is a (possibly infinite)
sequence of markings $(\mu_n)$, separated by elementary moves, and
connecting $\mu^-$ to $\mu^+$ (in the sense that $\mu_n$ is either
equal to $\mu^+$ for the last $n$, or converges to it as $n\to\infty$
if $\mu^+$ has a lamination component, and similarly for $\mu^-$).
Each marking is composed of curves that occur as vertices in a nested collection
of geodesics of $H$, which is known as a ``slice'' of $H$, and
successive markings are separated by elementary moves, which
correspond in a specific way to forward motion along the geodesics of
$H$.

\subsubsection*{Hierarchies and projections}

We will make crucial use of Lemma 6.2 in \cite{masur-minsky2},
sometimes called the ``large link'' lemma. Given markings $\mu^\pm$
and a subsurface $W\subset S$ we let $\hull_W(\mu^\pm)$ denote a
geodesic in $\CC(W)$ joining $\pi_W(\mu^+)$ to $\pi_W(\mu^-)$
(there may be more than one such geodesic but hyperbolicity implies
that all such choices are within uniform Hausdorff distance of each other).

\begin{lemma}{large link}{}
There exists $A=A(S)$ such that if
$H(\mu^\pm)$ is a hierarchy, 
$W\subset S$ is an essential subsurface, and \hbox{$d_W(\mu^+,\mu^-)>A$},
then $H(\mu^\pm)$ contains a geodesic $k_W$ with
domain $W$ and 
$$
d_{\rm Haus}(k_W,\hull_W(\mu^\pm)) \le A.
$$
Moreover
$$d_W(\tau_W,\mu^+)\le A\ \ \ \ {\rm and}\ \ \ \ \ d_W(\iota_W,\mu^-)\le A$$
where $\tau_W$ and $\iota_W$ are the terminal and initial vertices of $k_W$.
\end{lemma}

In fact the first inequality of Lemma \ref{large link} can be strengthened to something
that holds in the setting where a geodesic $k_W$ may not necessarily
exist: 
\begin{lemma}{curves near geodesic}
Given $S$ there exists $M$, such that for any pair of generalized
markings and any essential $W\subset S$, 
$$
d_{\rm Haus}\left(  \pi_W(\CH(\mu^\pm)), \  \hull_W(\mu^\pm) \right)
\le M.
$$
\end{lemma}
This result, which is established in the proof of Lemma 5.14 of
\cite{ELC1}, follows from Lemmas 6.1 and 6.9 in \cite{masur-minsky2},
which are part of the same machinery used in the proof of the large
link lemma.

\subsection{Model Manifolds}
\label{models}

To each hierarchy $H=H(\nu^\pm)$  we associate (in \cite{ELC1}) a {\em
  model manifold} $M=M(\nu^\pm)$, which is
  equipped with an orientation-preserving embedding into $S\times\R$ (which
  we treat as inclusion), 
  a path metric and a disjoint collection of
  {\em tubes}, one for each vertex of $H$.  The tube associated to
  $v\in\CH$ is an open solid torus of the form
  $U(v) \equiv \collar(v)\times I$ where $\collar(v)\subset S$ is an annulus whose core is $v$,
and $I$ is an interval (sometimes infinite). 
Each tube $U(v)$ is isometric to a standard Margulis tube (possibly
parabolic, for finitely many of the $v$).  Let $\UU\subset M$ denote
the union of all the tubes. The complement, $M\setminus \UU$,
decomposes into a union of {\em blocks}, which (with the exception of
a bounded number of {\em boundary blocks}) are submanifolds that
fall into a fixed finite number of isometry classes. The boundary of
each block is a union of annuli on tube boundaries and {\em level 3-holed spheres},
where the latter have the form $Y\times\{t\}$,
for a three-holed sphere $Y\subset S$ obtained as a complementary
component of $S\setminus\collar(\Gamma)$ for a curve system
$\Gamma$. (The boundary blocks, whose structure is slightly more
complicated, are all adjacent to the boundary of $M$, if any, and will
not affect the rest of our arguments).

The model contains a collection of {\em split-level surfaces}, each
associated to markings 
or partial markings that occur in resolutions.
Suppose $\mu$ is such a marking, restricted to a subsurface $W$ (so
that $\boundary W\subset \base(\mu)$ and $\base(\mu)$ determines a
pants decomposition of $W$).
The split-level surface $F_\mu\subset M\setminus\UU$ is
a disjoint union of level three-holed spheres $Y\times\{t_Y\}$,
where $Y$ runs over the components of $W\setminus
\collar(\base(\mu))$. Each three-holed sphere is properly embedded in
$(M\setminus\UU,\boundary\UU)$, and in the induced metric they are all
isometric to a single standard 3-holed sphere. Moreover, if $F_Y$ intersects
a tube $U(v)$, then $F_Y\cap U(v)$ is a geodesic in the metric on
$\partial U(V)$.

An {\em extended} split-level surface $\hhat F_\mu$ is obtained from
$F_\mu$ by adding, for every $v$ in $\base(\mu)\intersect int(W)$,  an annulus in
the corresponding tube $U(v)$. These annuli are identified with the
corresponding collars in a way that extends the identification of
$W\setminus\collar(\base(\mu))$ with $F_\mu$ to an identification of
$W$ with $\hhat F_\mu$. In particular
$\hhat F_\mu$ is an isotope of $W\times\{0\}$.

The annulus in each $U(v)$ is chosen so that it has a $CAT(-1)$
metric: If $U(v)$ is the Margulis tube with geodesic core then this
can be done by extending the boundaries of the annuli radially to the
core, and if $U(v)$ is parabolic we can simply rule the annulus by
geodesics connecting the boundaries. 

If the domain surface $W$ of an (extended) split-level surface is all
of $S$ we call it maximal.

\medskip

The maximal extended split-level surfaces $\hhat F_{\mu_n}$ associated to a
resolution are isotopes of $S\times\{0\}$ and are {\em monotonically
arranged} in the sense that the transition from $\hhat F_{\mu_n}$ to
$\hhat F_{\mu_{n+1}}$ always involves isotoping a subsurface upward in the
  $\R$ direction of $S\times\R$. This provides a connection between 
topological ordering in $M$ and the directionality of the hierarchy,
aspects of which we will state more precisely below.

\subsubsection*{Bilipschitz model map}

The main theorem of \cite{ELC2} provides a bilipschitz homeomorphism
between the model manifold associated to the end invariants of a
hyperbolic 3-manifold $N$, and the {\em augmented convex core} of $N$,
denoted $\hhat C_N$. This is the union of a 1-neighborhood of the
convex hull of $N$ with the thin part of $N$. (An extension of this
theorem gives a model that covers all of $N$, but we will not need
it). We give here a statement that combines this bilipschitz map with
other structural facts derived in that and related papers:

\begin{theorem}{bilipschitz model}{}
Given $S$, there exists $K_h>1,$ $\epsilon_h>0$  and $L_h>0$ such that,  if 
$\rho\in AH(S)$ has end
invariants $\nu^\pm$, then 
\begin{enumerate}
\item \label{model map}
There exists a $K_h$-bilipschitz orientation-preserving
homeomorphism
$h:M(\nu^+,\nu^-)\to \hhat C_{N_\rho}$,
\item \label{length bound}
$$\CH(\nu^\pm) \subset \CC(\rho,L_h).
$$
\item \label{short curves} 
$$\CC(\rho,\ep_h) \subset\CH(\nu^\pm)$$
\item \label{tubes to tubes} If  $l(\alpha)<\epsilon_h$, then 
$$h(U(\alpha))=\MT(\alpha).$$
\end{enumerate}
\end{theorem}

\medskip\noindent
{\bf Remark:} Part (\ref{length bound}) is a formal consequence of
(\ref{model map}), but is established (Lemma 7.9 in \cite{ELC1}) as
part of the proof of (\ref{model map}).

\medskip

An important additional feature of the model manifold is that, for any resolution $(\mu_n)$, 
every point of $M\setminus\UU$ lies within uniformly bounded
distance of at least one  split level surface $\hhat F_{\mu_n}$, since
every block intersects some $\hhat F_{\mu_n}$. 

\begin{lemma}{split level nearby}
There exists $c_0>0$ such that if $S$ is a compact surface,
$\rho\in AH(S)$ has end invariants $\nu^\pm$ with associated model
manifold $M=M(\nu^\pm)$ and $(\mu_n)$ is a resolution sequence of the associated
hierarchy $H=H(\nu^\pm)$, then if
$x\in M\setminus\UU$, there exists $n$ such that 
$$d(x,F_{\mu_n})<c_0.$$
\end{lemma}

\medskip

Let us also record the following useful fact, relating the appearance
of short curves in $N_\rho$ with high subsurface projections.
\begin{theorem}{kgccfact}{(Theorem B in \cite{minsky:kgcc})} 
Given a surface $S$,  $\epsilon>0$
and $L>0$, there exists $K=K(S,\epsilon,L)$ such that if $\rho\in AH(S)$ and $W$ is  an
essential subsurface of $S$, then $l_\rho(\partial W)<\epsilon$ if
${\rm diam}(\pi_W(C(\rho,L)))\ge K$.
\end{theorem}

\subsection{Ordering}
\label{ordering}

In a product $S\times \R$ there is a natural notion of topological
ordering induced by the projection $q:S\times\R \to \R$ to the second
factor. The details are however slightly messy so we take some care
with the definitions. 

Given two maps $f:A\to S\times\R$ and $g:B\to S\times\R$, we say that
$f$ lies {\em above} $g$ if $f$ extends to a map
$F:A\times[0,\infty)\to S\times\R$ such that $F(\cdot,0) = f$, the
  image of $F$ is disjoint from $g(B)$, and $q\circ F(\cdot,t)$ goes uniformly
  to $+\infty$ as $t\to+\infty$. We define {\em below} similarly with
  $+\infty$ replaced by $-\infty$.  If $g$ lies above $f$, $f$ lies
  below $g$, and the opposite statements are false, we write
  $f\topprec g$ (in spite of the notation, however, this relation is
  not a partial order). We will also apply this terminology to subsets
  of $S\times\R$ where the map is presumed to be the inclusion map. 

If $A$ and $B$ are subsets of $S$ and $f, g$ are homotopic to the
inclusions $A\to A\times\{0\}$, $B\to B\times\{0\}$, then we say that 
$f$ and $g$ {\em overlap} if $A$ and $B$ intersect essentially
(i.e. cannot be made disjoint by isotopy). Note that 
in this case if $f$ lies above $g$ then $g$ cannot lie above $f$, and
so on. If $f$ and $g$ are overlapping {\em level embeddings}, i.e. of the form
$a\mapsto (a,t)$ and $b\mapsto (b,s)$, then $f \topprec g$ if and only
if $t<s$. 

This notion of ordering can usefully be applied to the tubes in a
model manifold, where it is closely related to the ordering of the
geodesics in the hierarchy. (For an extensive discussion of topological ordering
and its relationship to the hierarchy, see sections 3 and 4 of \cite{ELC2}.)

Given a geodesic $g$ in $\CC(W)$, let $\pi_g$ denote the composition
of the projection $\pi_W$ with a nearest-point projection $\CC(W) \to
g$. 

For a directed geodesic $g$, we can fix an orientation-preserving
identification with an interval of $\Z$, so that addition makes
sense, and $a < b$ means $a$ occurs earlier than $b$.
This lemma describes the relation between
topological order of tubes in a model manifold, and the order of
projections along hierarchy geodesics. 

\begin{lemma}{top order and geodesic}
Let $H=H(\mu^\pm)$ be a hierarchy  and $M=M(\mu^\pm)$ the associated model
manifold. Suppose that $k$ is a geodesic in $H$, supported in a non-annular
$W\subset S$. There is a constant
$r=r(S)$ such that
\begin{enumerate}
\item
For any two vertices $u,v\in \CH(\mu^\pm)$ that overlap $W$ and each other,
$$
U(u) \topprec U(v) \implies    \pi_k(u) \le \pi_k(v) + r.
$$
\item
If $\gamma\in\CH(\mu^\pm)$ overlaps
a component $\beta$ of $\boundary W$, then $$
U(\beta) \topprec U(\gamma)  
\implies
d_W(\gamma,\mu^+) \le r
$$
and similarly
$$
U(\gamma) \topprec U(\beta)  
\implies
d_W(\gamma,\mu^-) \le r
$$
\end{enumerate}
\end{lemma}

\begin{proof} 
Fix a resolution $(\mu_n)$ of the hierarchy.
Following the notation in Section 4 of  \cite{ELC2}, for any vertex or simplex
$a$ in $\CH(\mu^\pm)$ define $J(a)\subset \Z$ to be the set of $n$ such
that $\base(\mu_n)$ contains $a$.  There is also a subset $J(k_W)$
which consists of those indices for which the geodesic $k_W$ is
``active'' in the resolution in a certain sense. Rather than give the
full definition we will note that 
$$J(k_W) \subset J([\boundary W]),$$ 
i.e. the geodesic is only active when $\boundary W$ is visible in the
marking, and that for each $n\in J(k_W)$ there must be some vertex $x$ 
of $k_W$ such that $x\in \base(\mu_n)$ -- in other words $n\in J(x)$. 
Lemma 4.9 of \cite{ELC2} states that $J(a)$ and $J(k_W)$ are {\em
  intervals in $ \Z$.}

Note that if $a$ and $b$ overlap then $J(a)$ and $J(b)$ are
disjoint. Because the split-level surfaces $F_{\mu_n}$ move
monotonically upward in $S\times\R$, we have immediately that 
\begin{equation}\label{topprec Jprec}
U(a) \topprec U(b) \implies \max J(a) < \min J(b).
\end{equation}
Another aspect of the monotonicity property of resolutions is that the
vertices of $k_W$ are traversed monotonically. That is, if
$u,v$ are vertices in $k_W$, then 
\begin{equation}\label{hW monotonic}
\max J(u) < \min J(v) \implies u < v.
\end{equation}

We will also need the following: If $n\in J(a)$, then
\begin{equation}\label{below J}
n < \min J(k_W) \implies  d_W(\mu_n,\mu^-) \le r_0
\end{equation}
and similarly
\begin{equation}\label{above J}
\max J(k_W) < n \implies  d_W(\mu_n,\mu^+) \le r_0
\end{equation}
for some uniform choice of $r_0$.
In other words, the projection to $W$ of everything in the hierarchy that
happens ``before'' $k_W$ is frozen, and similarly for everything
afterwards. This is a consequence of Lemmas 6.1 and 6.9 of
\cite{masur-minsky2}, in a way similar to Lemma \ref{curves near
  geodesic}.
(The discussion in \cite{masur-minsky2} identifies a certain sequence of
geodesic segments in $H$ that connect $k_W$ to $\mu^-$ and $\mu^+$ in such a
way that every vertex in the sequence projects nontrivially to
$\CC(W)$, and Lemma 6.1 shows that the parts of this sequence before
and after $k_W$, respectively, have bounded projections to $\CC(W)$.
Lemma 6.9 shows that every slice in the resolution meets some part of
this sequence.)

\medskip

{\em Proof of part (1):} Since $U(u) \topprec U(v)$, we can choose $s_u\in
J(u)$ and $s_v\in J(v)$ such that $s_u < s_v$ (by (\ref{topprec Jprec})).

If $s_u\in J(k_W)$, then  $\base(\mu_{s_u})$ contains a vertex
$u'$ in $k_W$, and in particular $s_u \in J(u')$. Note that $u'$ is within 1 of
$u$ in $\CC(W)$, and hence within 2 of $\pi_{k_W}(u)$.

If $s_v \in J(k_W)$ as well, then we similarly have $v'$ in $k_W$, so
that $s_v\in J(v')$ and $v'$ is within 2 of $\pi_{k_W}(v)$. It therefore suffices to
show that $u'\le v'+1$.

If $d_W(u',v') \le 1$ then we are done, and otherwise  $u'$ and
$v'$ overlap, so that $J(u')$ and $J(v')$ are disjoint. Since $s_u <
s_v$, it must be that $\max J(u') < \min J(v')$, so that $u'$ must
appear before $v'$ in $k_W$ (by (\ref{hW monotonic})). Again we are
finished in this case.  

If one of $s_u$ and $s_v$ is not in $J(k_W)$, suppose without loss of
generality it is $s_u$. 

If $s_u < \min J(k_W)$ then, by (\ref{below J}), $\pi_W(u)$ is within $r_0$ of the
initial point of $k_W$. In this case the conclusion holds trivially no
matter where $\pi_{k_W}(v)$ is. 

If $\max J(k_W)< s_u$, then by (\ref{above J}), $\pi_W(u)$ is within $r_0$ of the
final point of $k_W$. Since $s_u < s_v$ the same holds for $\pi_W(v)$,
and again we are done. 

\medskip

For the proof of part (2), we first note that since $\gamma,\beta\in\CH(\mu^\pm)$
and $\gamma$ and $\beta$ overlap, $U(\gamma)$ and $U(\beta)$ are topologically
ordered (see \cite[Lemma 4.9]{ELC2}). Moreover,
since $\beta\in [\boundary W]$ , $J(\gamma)$ and $J(k_W)$ are
disjoint. If $U(\beta) \topprec U(\gamma)$ then $\max J(k_W) < \min
J(\gamma)$, and as above, (\ref{above J}) implies that $d_W(\gamma,\mu^+)$ is
bounded. The proof of the opposite case is similar.
\end{proof}

\subsection{Topological lemmas}
\label{toplemmas}

In this section, we collect topological lemmas concerning ordering of curves and
surfaces in $S\times \R$, which will be applicable to
split-level and pleated surfaces in our hyperbolic manifolds. 

We begin by showing that a proper homotopy equivalence whose image is disjoint from a level curve
$\gamma_0=\gamma\times \{ 0\}$ lies above $\gamma_0$ if and only if
there is an essential curve on the surface which 
intersects $\gamma$ whose image lies above $\gamma_0$.

\begin{lemma}{above and below}
Let $\alpha$ and $\gamma$ be overlapping curves on $S$ and let
$$f:(S, \boundary S) \to (S  \times \R, \boundary S \times \R)$$
be a homotopy equivalence with image disjoint from $\gamma_0=\gamma
\times \{0\}$. Then $f|_\alpha$ lies above $\gamma_0$ if and only if
$f$ lies above $\gamma_0$. 
\end{lemma}

\begin{proof}
Clearly if $f$ lies above $\gamma_0$ then $f|_\alpha$ lies above $\gamma_0$. 

Suppose that $f|_\alpha$ lies above $\gamma_0$.
Let $A = \gamma \times (-\infty,0]$. We can homotope $f|_\alpha$,
in the complement of $\gamma_0$,
to a map whose image is disjoint from $A$. This homotopy can be
extended to all of $S$ where the homotopy is supported on a
neighborhood of $\alpha$ and the image of the homotopy is disjoint
from $\gamma_0$. Let $g:(S,\boundary S) \to (S \times \R, \boundary S
\times \R)$ be the new map. We can assume $g(S)$ intersects $A$
transversely. Then $\Gamma = g^{-1}(A)$ will be a collection of
disjoint curves on $S$. Since $g$ is a homotopy equivalence every
curve in $\Gamma$ will either be homotopic to  
$\gamma$ or will bound a disk. However, any curve that is homotopic to
$\gamma$ must intersect $\alpha$ and since $g(\alpha)$ is disjoint
from $A$ we must have that all curves in $\Gamma$ bound disks. Using
the standard innermost disk argument and the fact that $(S \times \R)
- \gamma_0$ is irreducible we can then homotope $g$, in the complement
of $\gamma_0$, to a map whose image is disjoint from $A$. Such a map
will lie above $\gamma_0$ so $g$ and therefore $f$ lies above
$\gamma_0$. 
\end{proof}

We next observe that a proper homotopy equivalence whose image is disjoint from an essential non-annular level
subsurface lies either above or below that subsurface

\begin{lemma}{above or below}
Let $W$ be a non-annular subsurface of a compact surface $S$.
If 
$$f:(S, \boundary S) \to (S  \times \R, \boundary S \times \R)$$
is a homotopy equivalence with image disjoint from $W_0=W \times \{0\}$,
then either $f$ lies above or below $W_0$.
\end{lemma}

\begin{proof}
Much as in the proof above, 
we may homotope $f$ (in the complement of $W_0$) so that
$f^{-1}(W\times\R) = f^{-1}(W\times(\R\setminus\{0\}))$ is a union of
essential subsurfaces of $S$. Since $f$ is a homotopy equivalence,
these subsurfaces must consist of one isotope of $W$ and a (possibly
empty) collection of
disjoint annuli.

Each annulus maps either to $W\times (0,\infty)$ or
$W\times(-\infty,0)$, where it must be homotopic rel boundary to
$\boundary W \times (0,\infty)$ or $\boundary W \times (-\infty,0)$,
respectively. Thus, after homotopy we may assume that
$f^{-1}(W\times\R)$ is just $W$, and $f(W)$ lies either in 
$W\times (0,\infty)$ or in $W\times(-\infty,0)$. It follows that we
can homotope $f$ to $+\infty$ or $-\infty$, respectively, in the
complement of $W_0$. 
\end{proof}

We say a map of a curve system in $S$ to $S\times\R$ is {\em unknotted} if it
is isotopic to a level embedding. 
It will also be useful to recall that knotting of the boundary is the
only obstruction to extending an embedding of an essential subsurface 
to a proper embedding of the entire surface which is isotopic to a level surface. 
This result is a special case of Lemma 3.10 in \cite{ELC2}, although
the proof of just this case is not hard.

\begin{lemma}{unknotted subsurface}
Let $W$ be a compact essential subsurface of $S$. If \hbox{$h: W \to S
  \times \R$} is an embedding homotopic to a level embedding, 
such that 
$h(\partial W)$ is unknotted, then $h$ extends to an embedding of $S$ in $S \times \R$ whose image 
is isotopic to $S \times \{0\}$.
\end{lemma}

Our final topological lemma is a degree computation which 
will be used to complete the proof of Theorem \ref{topistop} in the case
that the essential subsurface is an annulus. 

\begin{lemma}{both above}
Let $\gamma$ be an essential curve in $S$ and $\mathcal{N}(\gamma)$ be an open regular neighborhood of $\gamma \times \{1/2\}$ in 
$S \times [0,1]$ with boundary the torus $T$. Suppose that 
$$f:((S \times [0,1])\backslash \mathcal{N}(\gamma),\boundary S \times[0,1], T) \to ((S \times [0,1]) \backslash \mathcal{N}(\gamma), \boundary S \times[0,1],T)$$
is a continuous map of triples such that 
$f|_{S\times\{0\}}$ and $f|_{S\times\{1\}}$ are homotopic in the
complement of $\NN(\gamma)$.

Then the restriction of $f$ to $T$ has degree zero.
\end{lemma}

\begin{proof}
Since  $f|_{S\times \{0\} }$ and $f|_{S\times
  \{1\} }$ are homotopic to each other in the complement of
$\NN(\gamma)$, we may assume, possibly adjusting $f$ by homotopy, that $f(x,0) = f(x,1)$ for all $x\in S$.
Thus, $f$ descends to a map
$$F: (S \times S^1)\setminus \mathcal{N}(\gamma) \to (S \times
[0,1])\setminus \mathcal{N}(\gamma)$$ 
with $F(\boundary S \times S^1) \subseteq \boundary S \times [0,1]$.
Since $F$ defines a relative 3-chain in 
$$((S \times [0,1]) \backslash \mathcal{N}(\gamma), \boundary S \times [0,1])$$
whose boundary is $F|_T$, we see that
$$[F|_T]=0\in H_2((S \times [0,1])\backslash \mathcal{N}(\gamma), \boundary S \times [0,1]).$$
However, $[T]$ is a non-trivial homology class in 
\hbox{$H_2((S \times [0,1])\backslash \mathcal{N}(\gamma), \boundary S \times [0,1])$}
and 
$$F_*([T])=[F|_T]=d[T]$$
where $d$ is the degree of the restriction of $F$ to $T$ or, equivalently,
the degree of the restriction of $f$ to $T$. Therefore, this degree is zero.
\end{proof}

\subsection{Thick distance, bounded diameter lemmas, and subsurface
  product regions}
\label{thick product}

A simple but useful feature of hyperbolic geometry is the fact that
the thick part of a surface of bounded area has components of
uniformly bounded diameter. This, together with the observation that a
$\pi_1$-injective Lipschitz map of a hyperbolic surface into a
hyperbolic 3-manifold takes the thick part to the thick part (with
slightly different constants), is a useful tool that recurs, for
example, in the work of Thurston and Bonahon, and others. We develop
some notation in order to discuss and apply these ideas in our
context. 

If $X$ is a path-metric space and $A\subset X$ a subset, we denote by
$\thd_{A,X}(x,y)$ the infimum over paths $\alpha$ in $X$ from $x$ to $y$ of
the length of  $\alpha\intersect A$. This is a pseudometric
which assigns distance 0 to pairs of points in the same component of
$X\setminus A$. We let $\thdiam_{A,X}$ denote diameter with respect
to this pseudometric, and also use the abbreviation $\thdiam_A(X)
\equiv \thdiam_{A,X}(X)$.  It will also be useful, for a subset $Y\subset
X$, to let $\thd_{A,Y}$ denote the same as $\thd_{A\intersect
  Y,Y}$. 

We will use this notation when $X$ is a hyperbolic manifold $N$ and
$A=N_{thick(\ep)}$, and when $X$ is a model $M$ and $A = M\setminus\UU$.
In particular, it is easy to express the bounded diameter lemma for surfaces in this language.

\begin{lemma}{split BDL}
Given a compact surface $S$ and $\epsilon>0$, there exists $b=b(S,\ep)$ such that
\begin{enumerate}
\item
If  $X$ is a finite volume surface homeomorphic to $S$, 
then
$$
\thdiam_{X_{thick(\ep)}}(X) < b.$$

\item
If $M$ is a model manifold associated to a hierarchy and
$\hat F$ is an extended split-level surface in $M$, then
 $$
\thdiam_{\hat F_{thick(\ep)}}(\hat F) < b.
$$
\end{enumerate}
\end{lemma}

For hyperbolic surfaces, this is a standard consequence of the
thick-thin decomposition. For split-level surfaces, it follows from
the fact that each such surface is a union of three-holed spheres
whose metric is standard and a bounded number of  CAT($-1$) annuli
each of whose intersection with the  \hbox{$\ep$-thick} part consists of one or two annuli whose diameter
is uniformly bounded in terms of $\ep$.

The following remark will be useful for us. 
Let $f:X \to N$ be a \hbox{$\pi_1$-injective}
$K$-Lipschitz map. Then, since $f(X_{thin(\ep)}) \subset
N_{thin(K\ep)}$, we obtain:

\begin{equation}\label{thick diam bound}
\thdiam_{N_{thick(K\ep)},N}(f(X)) \le K\thdiam_{X_{thick(\ep)}}(X).
\end{equation}

Finally, let us also observe
\begin{lemma}{thick dist proper}
If $N$ is a hyperbolic 3-manifold and $N^0$ the complement
of cusp neighborhoods in $N$, then $\thd_{N_{thick(\ep)},N^0}$ is a
proper pseudometric on $N^0$, when $\ep$ is less than the Margulis constant.
\end{lemma}
This follows immediately from the fact that, with $\ep$ less than the
Margulis constant, the distance between any two components of
$N_{thin(\ep)}$ is uniformly bounded below.

\subsubsection*{Subsurface product regions}
Another important feature of Kleinian surface groups, and their
bilipschitz models, is the presence of ``thick product regions'',
namely regions in $S\times\R$ that are topologically products $W\times
J$, and geometrically anchored on Margulis tubes (or model tubes), and
bounded by split-level surfaces. We introduce some notation for
discussing these regions, and indicate their interaction with the
topological ordering relation and the structure of the hierarchy.

If $W$ is an essential non-annular subsurface of $S$ and $M$ is a model manifold
associated to $\rho\in AH(S)$, we say that $Q\subset M$ is a {\em $W$-product region}
if  there exists an orientation-preserving
homeomorphism $g:W\times [0,1]\to Q$ so that $g(\partial W\times [0,1])\subset U(\partial W)$
and, if $g_0:W\to Q$ is the inclusion map given by $g_0(x)=g(x,0)$, then $(g_0)_*$ is conjugate
to $\rho|_{\pi_1(W)}$. In this case, $\partial_0Q=g(W\times\{0\})$ and $\partial_1Q=G(W\times \{1\})$
are called the {\em horizontal} boundary components of $Q$.
Similarly, if $N=N_\rho$ and $h:M\to N$ is the model map, we say that $R\subset N$
is a  $W$-product region for $N$ if and only if $h^{-1}(Q)$ is a $W$-product region for $M$.
In this case, $\partial_0R=h(\partial_0Q)$ and $\partial_1R=h(\partial_1Q)$
are called the horizontal boundary components of $R$.

The following lemma shows that if one has a long geodesic $k_W\subset H$ associated to a level
subsurface, then one can find thick $W$-product regions in the model
manifold. 

\begin{lemma}{W cross sections}
Let $\rho\in AH(S)$ have associated hierarchy $H$ and model map
$h:M\to N_\rho$. 
Let $W\subset S$ be the support of a geodesic $k_W$ in $H$. 

For every simplex $v\in k_W$ there is an extended split-level surface
$\hhat F_v
\subset M$ in the isotopy class of $W$, 
passing through $U(v)$, such that, if $u,v\in k_W$ and $d_W(u,v) \ge 5$, then
\begin{enumerate}
\item $\hhat F_u$ and $\hhat F_v$ are
disjoint and comprise the horizontal boundaries of a $W$-product region
for $M$.
Moreover, if $u<v$ then, $\hhat F_u \topprec
  \hhat F_v$.
\item There exists $c_1 = c_1(S) > 0$ such that
$$\thd_{M\setminus\UU,Q}(F_u,F_v) > c_1 d_W(u,v).$$
\item Given $\ep>0$ there exists $c_2 = c_2(S,\ep)$ such that
for $R = h(Q)$ and $G_x = h(F_x)$, 
$$
\thd_{N_{thick(\ep)},R}(G_u,G_v) > c_2 d_W(u,v)
$$
\end{enumerate}
\end{lemma}

\begin{proof}
Each simplex
$v$ in $k_W$ can be extended to a marking $\mu(v)$ in $W$, and we can
let $\hhat F_v$ denote $\hhat F_{\mu(v)}$ with a slight abuse of
notation. For $u,v$ separated by at least 5, the pair $\hhat F_u$ and
$\hhat F_v$ form a special case of the ``cut systems'' described in
Section 4 of \cite{ELC2}.
Proposition 4.15 of \cite{ELC2}  implies that the surfaces
are disjoint, form the horizontal boundary of a $W$-product regions, and that the topological order
agrees with the ordering of vertices. This gives us (1).

To prove (2), we first observe that there is a definite lower bound $b_0$ on
separation between surfaces in the product region, namely
\begin{equation}\label{separate Fs}
\thd_{M\setminus\UU,Q}(\hhat F_u,\hhat F_v) > b_0 > 0
\end{equation}
when $d_W(u,v)\ge 5$. To see this, note that in $M$ every tube is
separated by a definite distance $b_1>0$ from every other tube -- this
is a consequence of the uniform geometry of the blocks that compose
$M\setminus\UU$.
Similarly, each level 3-holed sphere in a split-level surface has a
$b_2$-neighborhood which meets no other 3-holed spheres. We conclude
from this that the union of $\hhat F_u$ with all the tubes associated
to the base of its marking has a regular neighborhood of definite
width within $M\setminus \UU$. Since $d_W(u,v)\ge 5$, the two markings
cannot share any base curves, which implies that these regular
neighborhoods are disjoint. This gives (\ref{separate Fs}).

Now suppose that $d_W(u,v) \ge 5n$. Then by suitably subdividing the
interval $[u,v]$ in $k_W$ we 
we can subdivide their product region into a sequence of $n$ product
regions, each with the definite separation given by (\ref{separate Fs}). This suffices to give (2).

\medskip

Because $h$ is $K$-bilipschitz, we have the inequality
$$
\thd_{N_{thick(\ep_h)},R} \ge {1\over K}
\thd_{h^{-1}(N_{thick(\ep_h)}),Q}.
$$ 
Now since (by Theorem \ref{bilipschitz model})  $h^{-1}(N_{thick(\ep_h)})$ contains $M\setminus\UU$, we have
$$
\thd_{h^{-1}(N_{thick(\ep_h)}),Q} \ge \thd_{M\setminus\UU,Q}. 
$$
We conclude that 
$$
\thd_{N_{thick(\ep_h)},R}(G_u,G_v) > {1\over K}\thd_{M\setminus\UU,Q}(F_u,F_v)
$$
which gives (3). 

\end{proof}

We next observe that any large enough $W$-product region 
gives rise to a pants decomposition of $W$ consisting of hierarchy curves.

\begin{lemma}{pants hierarchy}
Given a compact surface $S$ and $\ep>0$, there exist $d_0=d_0(S,\ep)>0$ and
$d_1=d_1(S,\ep)>0$ with
the following properties. Let $\rho\in AH(S)$ have end invariants $\nu^\pm$ and associated
model manifold $M=M(\nu^\pm)$. Then, if $W$ is an essential non-annular subsurface of $S$,
$Q$ is a $W$-product region for $M$, 
\hbox{$z\in Q\setminus\UU$},
and 
$$\thd_{M_{thick(\ep)},Q}(z,W\times \{0,1\})>d_0,$$
then there exists a pants decomposition $\Gamma$ of $W$ so that
$\Gamma \subset \CH$ and \hbox{$U(\Gamma)\subset Q$.}

Similarly, if $h:M\to N=N_\rho$ is the model map, $R$ is a $W$-product region for $N$,
\hbox{$z\in R\cap N_{thick(\ep)}$} and 
$$\thd_{N_{thick(\ep)},R}(z,h(W\times \{0,1\}))>d_1,$$
then there exists a pants decomposition $\Gamma$ of $W$ so that
$\Gamma \subset \CH$ and \hbox{$h(U(\Gamma))\subset R$}.
\end{lemma}

\begin{proof}
We first prove our claim in the setting of the model manifold.
Lemma \ref{split level nearby} gives a maximal split-level surface
$F_{\mu}$ in $M$ whose image comes within $c_0$
of $z$.
Let $\Delta$ be the collection
of components of $\base(\mu)$ which are components of $\partial W$.
Let $Z$ be the component of  
$\hhat F_{\mu} \setminus U(\Delta)$ which comes within $c_0$ of $z$.

Since $Z$ is an extended split level surface it meets model tubes only
if they are associated to its base, and by construction no
base curve in $Z$ can be a component of $\boundary W$. Hence,
$Z$ cannot meet $U(\boundary W)$. Lemma \ref{split BDL} implies that
$$\thdiam_{Z_{thick(\ep)}}(Z)<b=b(S,\ep).$$
Therefore, if we choose $d_0>c_0+b$, then $Z\subset Q$.

We conclude from this that (the underlying subsurface associated to) $Z$ is
homotopic into $W$, which 
implies that $Z$ is isotopic to $W$ since $\boundary Z$ is
isotopic into $\boundary W$. This implies that $\base(\mu)$
contains a pants decomposition of $W$ and the lemma follows when $M$ is a model manifold.

We now assume we are in the hyperbolic manifold setting. Let $Q=h^{-1}(R)$.
If $z\in N_{thick(\ep)}\cap R$, then $h^{-1}(z)\subset M_{thick(\ep/K_h)}$ where $K_h$ is
the constant from Theorem \ref{bilipschitz model}. Therefore, there exists $z_0\in M\setminus\UU$
which may be joined to $h^{-1}(z)$ by a path in $M-U(\partial W)$ of length at most $c_1$
(where $c_1$ depends only on $S$ and on $\ep/K_h$). If
$$\thd_{N_{thick(\ep)},R}(z,h(W\times \{0,1\}))>d_1=K_hd_0(S,\ep/K_h)+K_h c_1,$$
then $z_0\in Q$ and
$$\thd_{M_{thick(\ep/K_h)},Q}(z_0,W\times \{0,1\})>d_0(S,\ep/K_h).$$
Then the model manifold case guarantees that there exists 
exists a pants decomposition $\Gamma$ of $W$ so that
$\Gamma \subset \CH$ and \hbox{$U(\Gamma)\subset Q$}. It follows that
\hbox{$h(U(\Gamma))\subset R$} and our proof is complete.

\end{proof}

\section{Controlled hierarchy curve systems}
\label{curve systems}

In this section we provide a tool for directly relating bounded-length curves
on pleated surfaces in a Kleinian surface group to the curves that
occur in the associated hierarchy. Lemma \ref{hierarchy system} says
that in any such pleated surface (or more generally a Lipschitz
surface with fixed bounds) there is a maximal collection of disjoint hierarchy curves, all
of uniformly bounded length, such 
that in the curve complexes of their complementary subsurfaces the projection of the entire
hierarchy is within bounded distance from the set of bounded-length curves.

\begin{lemma}{hierarchy system}
Given a compact surface $S$ and $K>0$ there exists \hbox{$B=B(S,K)>0$} such that if
$X \in \Teich(S)$ is a finite area hyperbolic surface,
$$f:X \to N$$
is a $K$-Lipschitz homotopy equivalence, $\rho=f_*\in {\mathcal D}(S)$
has end invariants $\nu^\pm$ 
and $H=H(\nu^\pm)$ is the associated hierarchy,
then there exists a curve system $\Gamma$ on $X$ such that if $\gamma\in \Gamma$,
then 
$$\gamma\in \CH \ \ \ \ {\rm and}\ \ \ \ \  l_X(\gamma)\le B.$$
Moreover,
if $W$ is a component of $ X\backslash \Gamma$ which is not a thrice-punctured
sphere, then
\begin{enumerate}
\item $\CC(W)$ contains no  curves in $\CH$, and
\item  there exists $\beta \in \CC(W)$ such that $l_X(\beta)\le B$ and 
$${\rm diam}(\pi_W(\beta\cup \CH))\le B$$
\end{enumerate}
\end{lemma}

The proof will proceed by contradiction. We assume that we have a sequence $\{\rho_n\}$ where it is not possible
to choose appropriate collections of hierarchy curves for any uniform choice of constants. We then re-mark
and pass to a subsequence so that there is a maximal collection $\Gamma_1$ of curves so that $l_{\rho_n}(\Gamma_1)\to 0$ and
if $Y_1$ is a component of $S-\Gamma_1$, then $\{\rho_n|_{\pi_1(Y_1)}\}$ is convergent. For large enough values of
$n$, $\Gamma_1$ are hierarchy curves in $H_n$. 
If $\lim\rho_n|_{\pi_1(Y_1)}$ is geometrically infinite, 
we pull-back a wide $Y_1$-product region from the corresponding 
end of the limit manifold and consider a split-level
surface passing through the middle of the product region in the approximates to find a pants decomposition of $Y_1$
by hierarchy curves for all
large enough values of $n$.
If $\lim\rho_n|_{\pi_1(Y_1)}$ is geometrically finite, then the set of
projections of bounded length curves  to any subsurface of $Y_1$ is
finite and our result follows as well.

\begin{proof}
We suppose that, for some $S$ and $K$, it is not possible to choose such a value of $B$ and proceed to find a contradiction. So assume there exists a sequence $\{f_n:X_n\to N_n\}$
of $K$-Lipschitz homotopy equivalences with associated representations $\{\rho_n=(f_n)_*\}$
and hierarchies $\{H_n\}$,
and $B_n\to\infty$ such that, for each $n$, one cannot find a disjoint collection of curves
in $H_n$ which have length at most $B_n$ on $X_n$  whose complementary regions
which are not thrice-punctured spheres have properties
(1) and (2) with constant $B_n$.

By passing to a subsequence and remarking the $X_n$ we can assume
that there is a curve system
$\Gamma_0$ on $S$ such that $\ell_{X_n}(\Gamma_0)\to 0$ and there is a uniform
lower bound on the length of any homotopically non-trivial, non-peripheral curve disjoint from 
$\Gamma_0$. We may further remark the $X_n$ by homeomorphisms of 
$S$ fixing $\Gamma_0$ to guarantee that if $\gamma$ is any fixed curve in $S\ssm\Gamma_0$,
then $\{\ell_{X_n}(\gamma)\}$ is bounded (where the bound depends on $\gamma$).
Notice that, by Theorem \ref{bilipschitz model}, each curve in $\Gamma_0$ eventually lies in the  set $\CH_n$
of hierarchy curves.
If each component of $S\ssm \Gamma_0$ is a thrice-punctured sphere, then we have
already achieved a contradiction.

We now focus on a component $Y$ of $S\ssm \Gamma_0$ which is not a thrice-punctured
sphere. Since each curve on $Y$ has bounded length on $X$, we may pass to a
subsequence so that 
$\{\rho_n|_{\pi_1(Y)}\}$ converges (up to conjugation). Let $\rho_Y\in AH(Y)$
be the limit of (the subsequence) $\{\rho_n|_{\pi_1(Y)}\}$.

Let $\Gamma_1$ be a maximal collection of disjoint simple closed curves on $Y$ such that
if $\gamma\in\Gamma_1$, then $\rho_Y(\gamma)$ is parabolic.
Notice that if $\gamma\in\Gamma_1$, then eventually $\gamma\in \CH_n$ and 
$\ell_{X_n}(\gamma)$ is bounded independent of $n$. We are again finished if
every component of $Y\ssm\Gamma_1$ is a thrice-punctured sphere.

Given a component $Y_1$ of $Y\ssm\Gamma_1$ which is not a thrice-punctured sphere,
we will construct a further subsequence and a curve system $\Gamma_2$
in $Y_1$  consisting of hierarchy curves in $\CH_n$ for all $n$
such that if 
$W$ is a component of $Y_1 \ssm \Gamma_2$ which is not a thrice punctured sphere, then $\CC(W)$ contains  no curves
in $\CH_n$ and if $\beta\in\CC(W)$ then there is an upper bound on
${\rm diam}(\pi_W(\beta\union\CH_n))$ for all $n$.
As we can apply this procedure iteratively to each component of $Y\ssm \Gamma_1$ which is
not a thrice-punctured sphere, this will achieve the desired
contradiction and complete the proof.

Fix, then,  a component $Y_1$ of $Y\ssm\Gamma_1$ which is not a thrice-punctured sphere.
We first claim that, for each $\alpha\in\CH_n$ which overlaps $Y_1$, 
there exists a curve $\psi_n(\alpha)$ in $\CC(Y_1)$
such that $\ell_{\rho_n}(\psi_n(\alpha))$ and
$d_{Y_1}(\psi_n(\alpha),\alpha)$ are uniformly bounded (independently
of $\alpha$ and $n$). 

Recall that $\ell_{\rho_n}(\alpha)\le L_0$ for all $\alpha\in \CH_n$ (see Theorem
\ref{bilipschitz model}). Hence we can let $\psi_n(\alpha)=\alpha$ if $\alpha$ is already in
$\CC(Y_1)$. In general, 
Lemma 4.1 of \cite{minsky:kgcc}
provides for each $\rho_n$ and each $\alpha\in \CH_n$ a pleated surface
$g:Z\to N_n$ (in the homotopy class of $f_n$, where $Z$ denotes a
hyperbolic structure on $S$) mapping $\boundary Y_1$ geodesically,
so that every minimal $Z$-length proper arc $\tau$ in $Y_1$ satisfies
a bound $d_{\AAA(Y_1)}(\tau,\alpha) \le c$, where $c$ depends only on $S$ and $L_0$.

Now given our upper bound on $\ell_{\rho_n}(\boundary Y_1)$
(and hence on $\ell_Z(\boundary Y_1)$), we
can combine arcs in $\tau$ and $\boundary \collar(\boundary Y_1)$ to
obtain an essential simple closed curve $\alpha'$ in $Y_1$ with 
$\ell_{\rho_n}(\alpha')\le \ell_Z(\alpha') \le d'$ and $d_{Y_1}(\alpha,\alpha') \le c'$, for
uniform $c'$ and $d'$.  This is the desired $\psi_n(\alpha)$. 

We can break into two cases now:

{\bf Case 1:}
There is a sequence $\alpha_n\in \CH_n$ 
such that $\alpha'_n = \psi_n(\alpha_n)$ takes on infinitely many values. 
After taking a further subsequence we can assume that
$\{\alpha'_n\}$ converge in $\PML(Y_1)$ to a lamination $\mu$.
Let $V$ be the support
of $\mu$. Using an argument of Kobayashi and Luo (see
\cite[\S4.3]{masur-minsky}), $d_V(\beta,\alpha'_n) \to \infty$, for
any fixed $\beta$,  and
hence, by Theorem \ref{kgccfact},  $\ell_{\rho_n}(\boundary V)$ goes to 0. Since
there are no further parabolics within $Y_1$, this means that $V=Y_1$,
and $\mu$ is filling in $Y_1$. Since $\ell_{\rho_n}(\alpha'_n)$ is
bounded while the length of $\alpha'_n$ in the fixed surface $Y_1$
goes to $\infty$, continuity of length \cite{Brockcont} implies that
$\ell_{\rho_Y}(\mu) = 0$, and 
therefore that $N_{Y}^0$ has a degenerate end with
base surface $Y_1$.

The idea now is that this degenerate end will correspond to a part of
the model manifold from which we can extract a pants decomposition of
$Y_1$ consisting of hierarchy curves. The details of this are a bit
delicate because we have to consider how the hierarchies $H_n$
interact with the structure of this limiting degenerate end. 

We may assume, after passing to a further subsequence, that $\{N_n\}$ converges geometrically to
$N_G$ and that there is a covering map $p:N_Y\to N_G$.
The Covering Theorem \cite{thurston-notes,canary-cover} can then be used to show
that there is a neighborhood  of this degenerate end which embeds in $N_G$ 
(see, for example, the proof of  Proposition 6.10 in \cite{ELC2}). Let $\cE$ be the image of
this neighborhood in $N_G$.
If $Y_1$ is identified with the interior of a compact surface $\bar Y_1$,
one may further assume that the closure of $\cE$ is homeomorphic to
$\bar Y_1 \times [0, \infty)$.
Let $\boundary_0\cE$  be the image of $\bar Y_1\times\{0\}$ and $\boundary_1\cE$ be the image of 
$\boundary\bar Y_1 \times [0,\infty)$ under this homeomorphism.

For any fixed   $R \subset \cE$ which is identified with $\bar Y_1\times [0,a]$ for some $a>0$,
for all large $n$
there exist 2-bilipschitz comparison maps 
$$\phi_n:R\to N_n$$
in the homotopy class of $\rho_n\circ(\rho_Y|_{\pi_1(Y_1)})^{-1}$ such
that 
$$\phi_n(R\intersect \boundary\MT(\partial Y_1))\subset\boundary
\MT(\partial Y_1,n)$$
where $\MT(\partial Y_1,n)$ is the collection of Margulis tubes  in $N_n$ associated to
the components of $\partial Y_1$.
(See Lemma 2.8 in \cite{ELC2}.)
Let $H_n$ and $M_n$ be the hierarchy and model manifold associated to $N_n$ and
let $h_n:M_n\to N_n$ be the model map. For all sufficiently large $n$, $\ell_{\rho_n}(\partial Y_1)<\ep_h$,
so $\partial Y_1\subset \CH_n$ and $\MT(\partial Y_1,n)=h_n(U(\partial Y_1)$.
Therefore, $R_n=\phi_n(R)$ is a $Y_1$-product region for $N_n$ for
all sufficiently large $n$.

Since the pseudometric 
$\thd_{(N_G)_{thick(\ep_1/4)},\EE}$ is proper 
(see Lemma \ref{thick dist  proper}),
we may choose $R$ so that there exists $z \in R\cap (N_G)_{thick(\ep_1)}$ with
$$\thd_{(N_G)_{thick(\ep_1/4)},R}(z,  Y_1 \times \{0,a\})) > 2d_1(S,\ep_1/2)$$
where $d_1(S,\ep_1/2)$ is the constant from Lemma \ref{pants hierarchy}. 
For large enough $n$, $z_n=\phi_n(z)\in N_{thick(\ep_1/2)}$ and
$$\thd_{(N_G)_{thick(\ep_1/2)},R_n}(z,  \phi_n(Y_1 \times \{0,a\})) > d_1(S,\ep_1/2).$$
Therefore, by Lemma \ref{pants hierarchy}, there exists a pants decomposition $\Gamma_n$
of $Y_1$ so that  $\Gamma_n\subset \CH_n$  and $h(U(\Gamma_n)) \subset R_n$. 
Since every curve in $\Gamma_n$ has a representative of uniformly bounded length in $R_n$,
it also has a representative of uniformly bounded length in $R$.
Since there are only finitely many such curves in $R$, we can pass to a subsequence such that $\Gamma_n$ 
is a fixed pants decomposition $\Gamma$. The fixed set of curves $\Gamma$ will have uniformly bounded length on $X_n$. 
This  completes the proof in this case.

{\bf Case 2:}
For some $k$, the union $\union_{n\ge k}\psi_n(\CH_n)$ is finite.
Taking a subsequence, we can assume that $\psi_n(\CH_n)$ is a constant
set $\Psi$.
Since $\CH_n\intersect \CC(Y_1)$ is
contained in $\Psi$, we may assume that it is constant for all $n$ as
well. 
Let $\Gamma_2$ be a maximal curve system in $Y_1$ whose elements are
in $\CH_n \intersect \CC(Y_1)$. 

Suppose first that $\Gamma_2$ is empty. Then there are no hierarchy
curves in $\CC(Y_1)$. The projection $\pi_{Y_1}(\CH_n)$  lies, for
all $n$, within a
uniform distance of the finite set 
$\psi$, and hence within uniform distance of any fixed curve
$\beta\in\CC(Y_1)$. This concludes the proof in this case. 

If $\Gamma_2$ is nonempty, consider any component $W$
of $Y_1\ssm \Gamma_2$. By maximality of $\Gamma_2$, $\CC(W)$ contains no hierarchy curves.
We can now repeat the argument replacing $Y_1$ by $W$. We construct a
new collection $\psi_n(\CH_n)$ in $\CC(W)$, and find ourselves either
in case 2 but with $\Gamma_2$ empty, or in case 1.

If it is case 2 we can complete the proof as above, with a uniform
bound on the set of projections of $\CH_n$ into $W$. If we are in case
1, we note that the argument shows that in fact the boundary of $W$
must consist of parabolics for $\rho_Y$, so that in fact $W=Y_1$,
contradicting the fact that $\Gamma_2$ is nonempty.
\end{proof}

\section{Projections of the bounded curve set}
\label{bounded curve set}

In this section, we prove Theorem \ref{bounded}, which we restate here
for convenience: 

\restate{Theorem}{bounded}{
Given $S$, there exists $L_0$ such that for all $L\ge L_0$ there exists
$D=D(S,L)$, such that given $\rho\in AH(S)$ with end invariants $\nu^\pm$ and an essential
subsurface $W\subset S$ which is not an annulus or a pair of pants, 
$$
d_{\rm Haus}\left(\pi_W(\CC(\rho,L)) , \hull_W(\nu^\pm(\rho))\right) \le D.
$$
Moreover, if $d_W(\nu^+(\rho),\nu^-(\rho)) > D$ then 
$\CC(\rho,L) \intersect\CC(W)$ is nonempty and 
$$
d_{\rm Haus}\left(\CC(\rho,L)\intersect\CC(W) , \hull_W(\nu^\pm(\rho))\right) \le D.
$$
}

The main new content of this theorem is the statement
that $\pi_W(\CC(\rho,L))$ is contained in a uniform neighborhood of
$\hull_W(\nu^\pm)$, and for this we use Lemma \ref{hierarchy system},
which gives us a comparison between the short curves in  pleated
surfaces and hierarchy curves.
The other inclusions were already known, and are essentially
consequences of the Bilipschitz Model Theorem \ref{bilipschitz
  model} which relates hierarchy curves to the hyperbolic structure,
and Lemma \ref{curves near geodesic} which controls hierarchies in
terms of subsurface projections. 

The main new ingredient in the proof is Lemma \ref{hierarchy system}. Given a curve
$\alpha\in\CC(\rho,L)$, we consider a pleated surface realizing $\alpha$ and the system $\Gamma$
of hierarchy curves produced by Lemma \ref{hierarchy system}. If some element of $\Gamma$ overlaps
$W$, then the result follows from Lemma \ref{curves near geodesic} which is essentially a version
of Theorem \ref{bounded} for hierarchy curves. If not, then  $\pi_W(\CH)$ has bounded diameter, and Lemma \ref{hierarchy system} provides a bounded
length curve whose projection to $W$ is uniformly near $\pi_W(\CH)$, which again allows us to complete the proof.

\begin{proof}
We first recall (from \ref{mu contained 1} in Section 2) that if $L_0$ is chosen to
be greater than the Bers constant $L_B$, then
\begin{equation}\label{mu contained}
\pi_W(\nu^\pm) \cap \overline{\pi_W(\CC(\rho,L))} \ne\emptyset
\end{equation}
where the closure is in the Gromov closure 
$\overline{\CC(W)} = \CC(W)\union\EL(W)$. 
From this we immediately have
\begin{equation}\label{bounded control mu}
\diam \pi_W(\CC(\rho,L)) \ge d_W(\nu^+,\nu^-).
\end{equation}
We next wish to get an inclusion in one direction, 
\begin{equation}\label{bounded contained}
\pi_W(\CC(\rho,L)) \subset \NN_{d_1} ( \hull_W(\nu^+,\nu^-) )
\end{equation}
for a uniform $d_1$. We recall that Theorem \ref{bilipschitz model} provides
$\ep_h>0$
such that $\CC(\rho,\ep_h) \subset\CH(\nu^\pm)$. Theorem \ref{kgccfact}
gives a constant $K=K(S,\ep_h,L)$ such that if 
$\diam \pi_W(\CC(\rho,L)) > K$, then $\ell_\rho(\gamma) <\ep_h$
for each component of $\boundary W$. Thus we 
suppose for now that $\diam \pi_W(\CC(\rho,L)) > K$, and
therefore that $[\boundary W] \subset\CH(\nu^\pm)$.

Given $\alpha\in\CC(\rho,L)$, let $f:X\to N_\rho$ be
a pleated surface, in the homotopy class of $\rho$, realizing
$\alpha$. Let $\Gamma$ be the curve system provided by Lemma
\ref{hierarchy system},  which consists of curves in $\CH(\nu^\pm)$
whose length in $X$ is at most $B=B(S,1)$.

If a component $\gamma\in \Gamma$ intersects $W$ essentially then,
since the length bound on $\gamma$ and $\alpha$ in $X$ implies a uniform
upper bound on the intersection number of $\gamma$ and $\alpha$ and hence
an uniform upper bound on $d_W(\alpha,\gamma)$ (see \cite[Lemma 2.1]{masur-minsky}),
we obtain a uniform upper bound on $d_W(\alpha,\CH(\nu^\pm))$. 
Lemma \ref{curves near geodesic} gives a uniform bound on the
Hausdorff distance between $\pi_W(\CH(\nu^\pm))$ and $\hull_W(\nu^\pm)$,
so we obtain a uniform upper bound  on $d_W(\alpha,\hull_W(\nu^\pm))$.

If, on the other hand, $W$ is disjoint from $\Gamma$, consider the component $Z$ of
$S-\collar(\Gamma)$ containing $W$. Since, by Lemma \ref{hierarchy system},
$\CC(Z)\cap\CH(\nu^\pm)$ is empty and $[\partial W]\subset\CH(\nu^\pm)$, we must
have $W=Z$.
Moreover, again by Lemma \ref{hierarchy system}, there exists $\beta\in\CC(W)$
such that $l_X(\beta)\le B$ and 
$$\diam_W(\beta\union \pi_W\CH(\nu^\pm)))\le B.$$
The 
length bounds on $\alpha$ and $\beta$ again give a uniform upper bound on
$d_W(\alpha,\beta)$, so we again obtain a uniform upper bound on
$d_W(\alpha,\hull_W(\nu^\pm))$.

Since, we have obtained a uniform upper bound on $d_W(\alpha,\hull_W(\nu^\pm))$
for all $\alpha\in\CC(\rho,L)$, we have established the containment
(\ref{bounded contained}), for some uniform $d_1$. 

Lemma \ref{large link} implies that there exists $A=A(S,L)$ 
such that if \hbox{$d_W(\nu^+,\nu^-) > A$}, then
the hierarchy $H$ contains a geodesic $k_W$ supported on the subsurface
$W$, whose initial and terminal vertices lie within $A$ of
$\pi_W(\nu^-)$ and $\pi_W(\nu^+)$. If 
$\diam(\pi_W(\CC(\rho,L)) > 2d_1 +A$, then (\ref{bounded contained})
implies that $d_W(\nu^+,\nu^-) > A$, and hence that we have $k_W$ in
$H$. 

In this case, assuming $L_0\ge L_h$, Theorem \ref{bilipschitz model} implies that the vertices of $k_W$ are all
contained in $\CC(\rho,L)$, and since the Hausdorff distance between $k_W$ and $\hull_W(\nu^\pm)$
is uniformly bounded (since $\CC(W)$ is Gromov hyperbolic),
we may conclude that there exists a uniform $d_2$ so that
\begin{equation}
\hull_W(\nu^\pm) \subset \NN_{d_2} (\CC(\rho,L)\cap\CC(W) ).
\end{equation}
In particular,
\begin{equation}
\hull_W(\nu^\pm) \subset \NN_{d_2} (\pi_W(\CC(\rho,L))).
\end{equation}
Therefore, if $\diam\pi_W(\CC(\rho,L)) >2 d_1 +K+ A$, then both 
the Hausdorff distance between $\hull_W(\nu^\pm)$ and $\pi_W(\CC(\rho,L))$
and the Hausdorff distance between $\hull_W(\nu^\pm)$ and $\CC(\rho,L)\cap\CC(W)$ are  bounded from
above by $d_1+d_2$. Therefore, we have established our theorem in this case if
we choose \hbox{$D=2d_1+d_2+A+K$}. 

It remains to consider the case that
$\diam\pi_W(\CC(\rho,L)) \le 2d_1+A+K$. However in this case
the conclusion of the theorem is immediate from (\ref{mu contained}).
\end{proof}

\section{Ending laminations in the algebraic limit}
\label{endlams}

We now prove Theorem \ref{limits}
which asserts that ending laminations of geometrically infinite ends arise
as limits of projections of end invariants.

\restate{Theorem}{limits}{
Let $\rho_n\to\rho$ in $AH(S)$.
If $W \subseteq S$ is an essential subsurface of $S$, other than an annulus or a pair of pants,
and $\lambda\in\EL(W)$ is a lamination supported on $W$, the following statements are
equivalent:
\begin{enumerate}
\item
$\lambda$ is a component of $\nu^+(\rho)$.
\item
$\{\pi_W(\nu^+(\rho_n))\}$ converges to $\lambda$
\end{enumerate}
Furthermore we have,
\begin{enumerate}[\indent (a)]
\item \label{converges} 
if $\{\pi_W(\nu^+(\rho_n))\}$ accumulates on $\lambda \in \EL(W)$ then it
  converges to $\lambda$,
\item \label{notincommon}
the sequences $\{\nu^+(\rho_n)\}$ and $\{\nu^-(\rho_n)\}$ do not converge to a common
$\lambda \in \EL(S)$, and
\item \label{iftopnotbottom}
if $W \subsetneq S$ is a proper subsurface then convergence of
  $\{\pi_W(\nu^+(\rho_n))\}$ to $\lambda \in \EL(W)$ implies
  $\{\pi_W(\nu^-(\rho_n))\}$ does not accumulate on $\EL(W)$.
\end{enumerate} 
The same statements hold with ``$+$'' replaced by  ``$-$''.
}

\medskip
We note that we allow the case $W = S$ unless explicitly noted
otherwise.  

For simplicity of notation, we let $\nu^+_n = \nu^+(\rho_n)$ and
$\nu^-_n = \nu^-(\rho_n)$.
If $\lambda$ is a component of $\nu^+(\rho)$, it is not difficult to
show that $\lambda$ is an accumulation point of either
$\{\pi_W(\nu^+_n)\}$ or of $\{\pi_W(\nu^-_n)\}$.  We first show, in
Lemma \ref{moreover}, that it cannot be both.  In order to show that
it is an accumulation point of $\{\pi_W(\nu_n^+)\}$, we use the
Covering Theorem and geometric limit arguments, to pull-back larger
and larger $W$-product regions from $N_\rho$ to the approximates. We
then consider intersections of split-level surfaces with these product
regions to find pairs of hierarchy curves in $\CH_n$, one of which
lies in a bounded set and the other of which approximates $\lambda$,
such that the geodesic representative of the approximation to
$\lambda$ lies above the geodesic representative of the curve in the
bounded set. This allows us to prove that (1) implies (2). On the
other hand, if $\{\pi_W(\nu_n^+)\}$ converges to $\lambda$, we check
that $\lambda$ is the ending lamination of a geometrically infinite
end of $N_\rho$. Then, using the fact that (1) implies (2), we see
that the end must be upward-pointing, which establishes that (2)
implies (1).

\begin{proof}
We first observe that if $\{\pi_W(\nu_n^+)\}$ converges to $\lambda$,
then $\{\pi_W(\nu_n^-)\}$ cannot accumulate at $\lambda$.

\begin{lemma}{moreover}{}{}
  Suppose that $\{\rho_n\}$ is a convergent sequence in $AH(S)$.  If
  $W \subseteq S$ is a (not necessarily proper) subsurface of $S$,
  $\{\pi_W(\nu_n^+)\}$ converges to $\lambda\in\EL(W)$, then
  $\pi_W(\nu_n^-)$ does not accumulate at $\lambda$.  Similarly, if
  $\{\pi_W(\nu_n^-)\}$ converges to $\lambda\in\EL(W)$, then
  $\{\pi_W(\nu_n^+)\}$ does not accumulate at $\lambda$.
\end{lemma}

\begin{proof}{}
  Let $W$ be any non-annular subsurface of $S$ that is also not a
  three-holed sphere.  Suppose that $\{\pi_W(\nu_n^+)\}$ converges to
  $\lambda\in\EL(W)$.  If $\{\pi_W(\nu_n^-)\}$ also accumulates at
  $\lambda$, then we can pass to a subsequence, still called
  $\{\rho_n\}$, such that both $\{\pi_W(\nu_n^+)\}$ and
  $\{\pi_W(\nu_n^-)\}$ converge to $\lambda$. Let $\rho=\lim \rho_n$.

  Let $\alpha$ be a curve on $W$. Then the distance of $\alpha$ to any
  geodesic joining $\pi_W(\nu_n^+)$ and $\pi_W(\nu_n^+)$ diverges to
  $\infty$.  Since there exists some $L\ge L_0$ such that
  $l_{\rho_n}(\alpha)\le L$ for all $n$, this contradicts Theorem
  \ref{bounded}.

The proof of the other case is exactly the same.
\end{proof}

Suppose that $\lambda\in\EL(W)$ is a component of $\nu^+(\rho)$ --
that is, $N_\rho^0$ has an upward-pointing end with base surface $W$
and  ending lamination $\lambda$.  In order to show that (1) implies (2),
it suffices to show that any subsequence
of $\{\rho_n\}$  has a further subsequence 
$\{\rho_{n_k}\}$ such that $\pi_W(\nu_{n_k}^+) \to \lambda$. 

Given any subsequence of $\{\rho_n\}$,  we may pass to a subsequence (still denoted $\{\rho_n\}$) such that
$\{ N_n\}=\{N_{\rho_n}\}$ converges geometrically to a manifold
$N_G$, which is covered by $N_\rho$. It is a consequence of the Covering Theorem \cite{thurston-notes,canary-cover}
(see Proposition 5.2 in \cite{AC-cores})
that there is a neighborhood of the geometrically infinite end of $N^0_\rho$
with ending lamination $\lambda$ which embeds in $N_G$. Let $\cE$ be the image of this neighborhood
in $ N_G$.
We may identify $\cE$ in an orientation-preserving way with $W\times[0,\infty)$.

After passing to a subsequence, we may assume that there exist 2-bilipschitz maps
$$\phi_n : W \times [0,n+1] \to N_n$$
so that $\phi_n(\partial W\times [0,n])\subset \MT(\partial W,n)$ where
$\MT(\partial W,n)$ is the collection of Margulis tubes in $N_n$ associated to the curves in $\partial W$.
After passing to a further subsequence we can adjust the product structure on $\cE$ and choose points
$$z_n \in  (N_G)_{thick(\ep_1)}\cap (W \times [n,n+1])$$
so that $\phi_n(W \times [k,k+1])$ is a $W$-product region,
$\phi_n(z_k)\in (N_G)_{thick(\ep_1/2)}$, and
$$\thd _{N_{thick(\ep_1/2)},\phi_n(W\times [k,k+1])}(\phi_n(z_k),\phi_n(W\times\{k,k+1\}))>d_1(S,\ep_1/2)$$
for all $k=0,\ldots, n$. In other words,
$(\phi_n(W \times [k,k+1]), \phi_n(z_k))$ satisfy the assumptions of Lemma \ref{pants hierarchy}. (See the proof of Lemma \ref{hierarchy system}.)

Let $h_n:M_n\to N_n$ be the model map provided by Theorem \ref{bilipschitz model}.
Lemma \ref{pants hierarchy} guarantees that for all $n$ there exists $\alpha_n \in\CH_n$ with 
\hbox{$h_n(U(\alpha_n)) \subset \phi_n(W \times [0,1])$}. 
Each $\alpha_n$ has a representative  in $W \times [0,1]\subset \cE$ of uniformly
bounded length. As in the proof of Lemma \ref{hierarchy system} there will be a finitely many such curves and we can pass to a further subsequence such that $\alpha_n$ is a fixed curve $\alpha$.

Similarly , for all $n$, Lemma \ref{pants hierarchy} implies that there exists a curve $\beta_n \in\CH_n$ with 
$h_n(U(\beta_n)) \subset \phi_n(W \times [n,n+1])$. Then $\{\phi_n^{-1}(\beta_n)\}$ is a sequence of bounded length curves
exiting $\cE$, so we have that $\beta_n \to \lambda$. 
In particular, for large $n$, $\alpha$ and $\beta_n$ overlap.
Furthermore, by construction $\phi^{-1}_n(h_n(U(\beta_n)))$ lies above 
$\phi^{-1}_n(h_n(U(\alpha)))$ in $\cE$.
Lemma \ref{unknotted subsurface} implies that $\phi_n(W\times \{1\})$ extends to an embedded surface $X$
isotopic to a level surface in $N_n$. Since $h_n(U(\alpha_n))$ and $h_n(U(\beta_n))$ lie in
a collar neighborhood of $\phi_n(W\times \{1\})$ we may assume they are disjoint from $X$.
Since each $\phi_n$ is orientation-preserving, this implies that
$$h_n(U(\alpha))\topprec h_n(U(\beta_n))$$
in $N_n$. Since $h_n$ is orientation-preserving, we may conclude that
$$U(\alpha)\topprec U(\beta_n)$$
in $M_n$.

Since $d_W(\alpha,\beta_n)\to\infty$ and $\alpha,\beta_n\in \pi_W(C(\rho_n,L_0))$,
Theorem \ref{bounded} implies that $d_W(\nu_n^+,\nu_n^-)\to\infty$.
We may pass to a subsequence so that \hbox{$d_W(\nu_n^+,\nu_n^-)\ge A(S)$} for all $n$, so
Lemma \ref{large link} implies that $H_n$ 
contains a geodesic $k_n$ with domain $W$. 
Lemma \ref{curves near geodesic}
implies that  $d_W(\pi_{k_n}(\alpha),\alpha)$ and 
$d_W(\pi_{k_n}(\beta_n),\beta_n)$ are uniformly bounded, where
$\pi_{k_n}$ is the projection from $\CC(S)$ to $k_n$ through
$\CC(W)$. 
After identifying $k_n$ with an interval,
Lemma \ref{top order and geodesic} implies that there exists $r$ so that
$$\pi_{k_n}(\alpha)<\pi_{k_n}(\beta_n)+r.$$

Now, since $\beta_n\to\lambda\in\EL(W)$,  it follows
that $\pi_{k_n}(\beta_n)\to\lambda$. Since $\pi_{k_n}(\alpha)<\pi_{k_n}(\beta_n)+k$,
$\beta_n$ lies between $\alpha$ and the terminal vertex $\tau_n$ of $k_n$ for all large
enough $n$.
Recalling that all the $\pi_{k_n}(\alpha)$ lie in a finite diameter set in $\CC(W)$,
we may conclude that the terminal vertex $\tau_n$ also converges to $\lambda$.
(Here, we use that $\CC(W)$ is hyperbolic and Klarreich's theorem \cite{klarreich}
identifying $\partial\CC(W)$ with $\EL(W)$.)
Since, by Lemma \ref{large link}, $d_W(\tau_n,\nu_n^+)$ is uniformly bounded, we
further conclude that $\{\pi_W(\nu_n^+)\}$ converges to $\lambda$, as desired. 
This completes the proof that (1) implies (2).

\medskip

Now suppose that 
$\{\pi_W(\nu_n^+)\}$ converges to $\lambda\in\EL(W)$. 
Lemma \ref{moreover} then implies that $\{\pi_W(\nu_n^-)\}$ does not accumulate
at $\lambda$. Therefore, $d_W(\nu_n^+,\nu_n^-)\to \infty.$
Lemma \ref{large link} implies that, for all large enough $n$, $H_n$ contains a geodesic $k_n$ with
base surface $W$.
For each $n$ choose a vertex $\beta_n$ of $k_n$ so that
$\beta_n\to \lambda$. 
By Klarreich's theorem,   there is a subsequence so that $\{\beta_n\}$
converges projectively to  
a measured lamination $\nu$ on $W$ whose support is $\lambda$ -- that is,
$\{\beta_n/l_X(\beta_n)\}$ converges to $\nu$ where $X$ is a fixed finite area hyperbolic
metric on $W$. By continuity of length \cite{Brockcont}, 
$l_\rho(\nu)=\lim l_{\rho_n}(\beta_n)/l_X(\beta_n) = 0$, so
$\lambda$ is unrealizable in $\rho$. 
This implies that $\lambda$ is an ending lamination for an end based
on $W$ (see \S\ref{back}).

If this end were
downward-pointing then, since (1) implies (2) (applied in the downward-pointing case),
$\{\pi_W(\nu_n^-)\}$ would converge to $\lambda$, which would contradict
Lemma \ref{moreover}. Therefore,  the end must be upward-pointing. This completes the proof
that (2) implies (1).

In order to establish Claim~(\ref{converges}), we assume that
$\{\pi_W(\nu^+_n)\}$ accumulates at $\lambda\in\EL(W)$. By applying implication (2) implies
(1) to a subsequence $\{\rho_{n_j}\}$ where $\{\pi_W(\nu^+_{n_j})\}$ converges to $\lambda$, we see that $\lambda$ is a component of $\nu^+(\rho)$, and 
so, applying  implication (1) $\implies$ (2), we see that the entire sequence
$\{\pi_W(\nu^+_n)\}$ converges to $\lambda$.
Claim~(\ref{notincommon}) in the statement follows
immediately from Lemma \ref{moreover}.  Finally, for
claim~(\ref{iftopnotbottom}), note that for $W$ a proper subsurface of
$S$, if $\{\pi_W(\nu_n^+)\}$ converges to $\lambda^+ \in \EL(W)$ and
$\{\pi_W(\nu_n^-)\}$ converges to $\lambda^- \in \EL(W)$, then
$\lambda^+$ is a component of $\nu^+(\rho)$ and $\lambda^-$ is a
component of $\nu^-(\rho)$ by an application of (2) $\implies$ (1).
Therefore, for some boundary component $\gamma$ of $W$ that is
non-peripheral in $S$, $\gamma \times \{0\}$ is isotopic in $S \times
I$ into the annuli $P^+$ and $P^-$ determined by the relative compact
core for $N_\rho$, contradicting that no two components of $P^+$ and
$P^-$ are isotopic.
\end{proof}

\section{Overlapping curves}
\label{overlap}

In this section we prove Theorem \ref{topistop}, which
states that if a curve $\alpha\in\CC(\rho,L)$ overlaps and lies above $\boundary W$ in  $N_\rho$,
then $\pi_W(\alpha)$ is uniformly close to $\pi_W(\nu^+(\rho))$.

\restate{Theorem}{topistop}{
Given $S$ and $L>0$ there exists $c$ such that, given
$\rho\in AH(S)$, an essential subsurface
$W\subset S$ which is not  a pair of pants, 
and a curve $\alpha\in\CC(\rho,L)$ such that
$\alpha^*$ lies above the geodesic representative of some
component of $(\partial W)^*$ that it overlaps, then
$$d_W(\alpha,\nu^+(\rho))\le c.$$
Furthermore, if $W$ is not an annulus
or a pair of pants, $\alpha\in\CC(\rho,L)$ overlaps $\partial W$, and
$$ d_W(\alpha,\nu^-) > c$$
then $\alpha^*$ lies above the geodesic representative of every component of $\partial W$
that it overlaps.

The same holds when replacing ``above'' with ``below'' and $\nu^+$
with $\nu^-$.
}

When $\rho(\alpha)$ or $\rho(\beta)$ is parabolic we interpret the statement
``$\alpha^*$ lies above $\beta^*$'' to mean that, all sufficiently
short representatives of $\alpha$ in $N_\rho$ lie above
all sufficiently short representatives of $\beta$. 
Equivalently, $\alpha^*$ lies above $\beta^*$
if $\alpha$ is in the top end invariant or $\beta$ is in the
bottom end invariant.

In the course of the proof we will notice that if $l_\rho(\alpha)\le L$, $W$ is not an annulus,
and $d_W(\nu^+,\nu^-)$ is sufficiently
large, then $\alpha^*$ either lies above or below all geodesic
representatives of components of $\partial W$ which it overlaps
(see Proposition \ref{curve above or below}).

\medskip

We will first prove the theorem when $W$ is not an annulus. The result
will be straightforward if $d_W(\nu^+,\nu^-)$ is small. 
We sketch the argument when
$d_W(\nu^+,\nu^-)$ is large and therefore all the components of
$\partial W$ lie in $\CH$, are very short, and 
Lemma \ref{W cross sections} gives us a wide $W$-product region.
If $\alpha$ is a curve in the hierarchy
the theorem follows from statement (2) of Lemma \ref{top order and
  geodesic}. If not we may realize $\alpha$ by a pleated surface $f:X
\to N$ and then replace $\alpha$ with a hierarchy curve $\gamma$ that
overlaps $W$ given to us by Lemma~\ref{hierarchy system}. If
$\alpha^*$ overlaps and lies above a component $\beta^*$ of $(\boundary W)^*$ then the
pleated surface also lies above $\beta^*$. Therefore, if $\gamma$ and
$\beta^*$ overlap, $f(\gamma)$ also lies above $\beta^*$ and
the theorem again follows from (2) of Lemma \ref{top order and
  geodesic}.
Most of the work in the proof involves the case when
there is no such $\beta^*$ that both $\alpha$ and $\gamma$ overlap.
For example, it is possible that $\gamma$ lies in
$W$. Here we employ the observation that the pleated surface has
bounded penetration into the wide $W$-product region, and an
application of part (1) of Lemma \ref{top order and geodesic} completes the
proof.

\medskip

In the case that $W$ is an annulus, note that the argument above works
perfectly well as long as the curve $\gamma\in\Gamma$ overlaps the core of
$W$. However, the possibility that $\gamma$ is equal to the core of
$W$, together with the phenomenon of wrapping, force
a completely different approach.  Given a curve $\alpha$ of bounded
length such that $\alpha^*$ lies above $\gamma^*$ (where $\gamma$ is
now the core of $W$), we first note that there does exist a curve $\beta$
of bounded length such that $\beta^*$ also lies above $\gamma^*$, and
that moreover $d_\gamma(\beta,\nu^+)$ is bounded ($\beta$ is
essentially obtained from the top ending data of $N_\rho$, in Lemma
\ref{nearby curve}).  We therefore have to bound
$d_\gamma(\alpha,\beta)$. To do this we consider a ``model manifold''
constructed with ending data $\alpha$ and $\beta$. If
$d_\gamma(\alpha,\beta)$ is large enough then this model will have a
deep tube $U(\gamma)$, and a Lipschitz model map constructed in Lemma
\ref{special model} using a variation on the work in \cite{ELC1} will
take $\boundary U(\gamma)$ to the boundary of the corresponding Margulis tube
$\MT(\gamma)$. The fact that $\alpha^*$ and $\beta^*$ are
both above this tube will be used to show that the map $\boundary
U(\gamma) \to \boundary \MT(\gamma)$ has degree 0, but a Lipschitz map
of degree 0 between tori of bounded geometry cannot have kernel
generated by a very long curve (Lemma \ref{lipschitz_torus}). This
bounds the length of the meridian of $U(\gamma)$, and hence bounds
$d_\gamma(\alpha,\beta)$.

\subsection{Proof in the non-annular case}
\label{overlap nonannular}

When $W$ is not an annulus the theorem will follow from this proposition:
\begin{proposition}{curve above or below}
 If $S$ is a compact surface and $L>0$, then there exists $c_2=c_2(S,L)$ such that
 if $\rho\in AH(S)$,
 $W$ is a non-annular subsurface, 
 $$d_W(\nu^+,\nu^-)>c_2$$
 and $\alpha\in\CC(S,L)$
 overlaps $\partial W$, then either $\alpha^*$ lies above the geodesic representative in $N_\rho$ of
 every component of $\partial W$ it overlaps or  $\alpha^*$ lies below the geodesic representative of
 every component of $\partial W$ it overlaps. 
 
 Moreover, if $\alpha^*$ lies above $(\partial
 W)^*$ then 
$$d_W(\alpha, \nu^+) \le c_2.$$
The same holds when replacing ``above'' with ``below'' and ``+'' with ``-''.
 \end{proposition}

\begin{proof}[Proof of \ref{topistop} in non-annular case, given Proposition \ref{curve above or below}:]
If $d_W(\nu^+,\nu^-) > c_2$, then the second claim of Proposition \ref{curve above or
  below} is exactly the first claim of the Theorem. For the second claim of the theorem we note that the first claim implies that  $\alpha^*$ cannot lie below $(\boundary W)^*$. Proposition \ref{curve above or below} then implies that $\alpha^*$ lies above $(\boundary W)^*$.
  
If $d_W(\nu^+,\nu^-) \le
c_2$ it is convenient to assume, without loss of generality, 
that $L\ge L_0$ where $L_0$ is the constant from Theorem
\ref{bounded}.
Theorem \ref{bounded}  then implies, since $l_\rho(\alpha)\le L$, that
$\pi_W(\alpha)$ lies within $D=D(S,L)$ of  $\hull_W(\nu^\pm)$. 
Therefore,
\begin{eqnarray*}\label{consequence of bounded}
d_W(\alpha,\nu^+) &\le& D + d_W(\nu^+,\nu^-)\\
& \le & D +c_2
\end{eqnarray*}
regardless of whether or not $\alpha^*$ lies above any component of $(\partial W)^*$. In particular, we have both claims of the theorem. \end{proof}

\begin{proof}[Proof of \ref{curve above or below}]
We first assume that
$$c_2 > A,$$
where $A=A(S)$ is the constant  given in Lemma
\ref{large link}. Therefore, the assumption that $d_W(\nu^+,\nu^-) > c_2$ implies that $W$ is the
support of  a geodesic $k_W$ in the hierarchy  $H=H(\nu^+,\nu^-)$.

We  may further  assume that 
$$c_2>K+2D,$$
where $K=K(S,\ep_h,L)$ is the constant
from Theorem \ref{kgccfact} and $D=D(S,L)$ is the constant from Theorem \ref{bounded}.
Then Theorem \ref{bounded} implies
\hbox{$\diam_W(\CC(\rho,L))>K$}, 
so that \hbox{$l_\rho(\partial W)<\ep_h$}.

Let $N=N_\rho$ and let $h:M\to N$ be the model map provided by
Theorem \ref{bilipschitz model}. Since  $l_\rho(\partial W)<\ep_h$,
$\partial W\subset\CH$ (where $H=H(\nu^\pm)$) and
$$h(U(\partial W))=\MT(\partial W).$$

If $d_W(\nu^+, \nu^-) = \infty$ then $W$ supports a geometrically infinite end and every component of $(\partial W)^*$ will be parabolic and will either lie in $\nu^+$ or $\nu^-$. The ordering claim of the proposition then follows. For the remainder of the proof we assume that if $W$ supports a geometrically infinite end it is downward-pointing and therefore $k_W$, the hierarchy geodesic supported on $W$, has a terminal vertex $\tau_W$ and $\alpha^*$ lies above every component of $\partial W$ that it overlaps.

Let $f:X\to N$ be a pleated surface, in the homotopy class
of $\rho$, realizing $\alpha$. 
Let $\Gamma$ be the system of hierarchy curves and $B=B(S,1)$ the constant provided by
Lemma \ref{hierarchy system} which bounds the length on $X$ of every curve in $\Gamma$.
Since we know that $W$ is the support of a geodesic
in $H$, and there are no hierarchy curves in the complement of
$\Gamma$,  there exists a hierarchy curve $\gamma\in \Gamma$ which
overlaps $ W$. 
Notice that,  since $l_X(\alpha)\le L$ and $l_X(\gamma)\le B$,
$d_W(\alpha,\gamma)\le a$ for some uniform constant $a$.
Therefore,
\begin{equation}\label{reduction to t_W}
d_W(\alpha,\nu^+) \le d_W(\gamma,\tau_W)+a+A.
\end{equation}
Thus our goal now is to bound $d_W(\gamma,\tau_W)$.

We will define a constant $a_1$ and require that 
$$c_2>2a_1+2A,$$
so that, by Lemma \ref{large link}, $k_W$ has length at least $2a_1$.
The constant $a_1$ will be chosen so that the $W$-product region provided by
Lemma \ref{W cross  sections} will be ``thick'' enough for our purposes.

We begin by giving the argument in a simpler case, where it is easier to understand the 
structure of the argument.

\subsubsection*{Simplified ordering argument}
Consider the case in which the following hold:
\begin{enumerate}
\item[(S1)] $f(X)$ is disjoint from $\MT(\boundary W)$,
\item[(S2)] $f(\gamma) \subset h(U(\gamma))$.
\end{enumerate}

Recall from Lemma \ref{split BDL} that there exists a constant $b=b(S,\ep_1)$,
so that 
$$\thdiam_{X_{thick(\ep_1)}}(X)\le b.$$
Since $f$ is 1-Lipschitz it follows, see (\ref{thick diam bound}), that
$$\thdiam_{N_{thick(\ep_1)}}(f(X))\le b.$$

Let $c_3=c_3(S,\ep_1)$ be the constant given by Lemma \ref{W cross sections}, 
and assume that $a_1$ has been chosen so that
$$a_1\ge\lceil b/c_3 \rceil.$$

Let $v_0<v_1<v_2=\tau_W$ be three vertices in $k_W$ such that \hbox{$d_W(v_i,v_{i+1}) = a_1$}.
Lemma \ref{W cross  sections} gives us images of extended split-level surfaces
$G_i = h(\hhat F_{v_i})$ and two $W$-product regions $R_1$ and $R_2$ where $R_1$
has horizontal boundary $G_0\cup G_1$ and $R_2$ has horizontal boundary $G_1\cup G_2$.
If we let $R=R_1\cup R_2$, then
$G_1 \subset R$, and
\begin{equation}\label{F1 Fj thick far simp}
\thd_{N_{thick(\ep_1)},R}(G_1,G_j) > b
\end{equation}
for $j=0$ and $j=2$.  

We claim now that $f(X)$ is disjoint from $G_1$. 
If not, then, since $f(X)$  is disjoint from $\MT(\partial W)$ and
$$\thdiam_{N_{thick(\ep_1)}}(f(X))\le b,$$
$f(X)$ would have to be contained in $R$. However, this is impossible, since
$f$ is a homotopy equivalence.

Since $f(X)$ is disjoint from $G_1$, it lies above it or below it by
Lemma \ref{above or below}. By Lemma \ref{above and below}, $\alpha^*$ lies either above or below every component of $(\partial W)^*$ that it overlaps. This  proves the ordering statement of the proposition.

Assume that the former holds, that $f(\alpha) = \alpha^*$ lies above the core curve
of $\MT(\beta)$ for every component $\beta\subset \boundary W$ that $\alpha$ overlaps. Since $f(\alpha)$ is disjoint from the tube $\MT(\beta)$, it  also lies above the corresponding
boundary component of $G_1$. Hence $G_1 \topprec f(X)$. 
We therefore can conclude that $G_1 \topprec f(\gamma)$, and finally,
since $f(\gamma)\subset h(U(\gamma))$, that
\begin{equation}\label{G1 below gamma simp}
G_1 \topprec h(U(\gamma))
\end{equation}
provided $G_1$ is disjoint from $h(U(\gamma))$. 

Note that $G_1$ intersects $h(U(\gamma))$ only if $\hhat F_{v_1}$
intersects $U(\gamma)$. Since $\hhat F_{v_1}$ is an extended
split-level surface, this can only happen if $\gamma$ is one of the
base curves of $\hhat F_{v_1}$, and in particular if $\gamma$ is
disjoint from $v_1$ (as curves on $S$).  In this case 
$d_W(\gamma,v_1)\le 1$, so 
$$d_W(\gamma,\tau_W)\le a_1+1.$$

If $\gamma$ does intersect $v_1$, then $G_1$ and $h(U(\gamma))$ are
disjoint and so
$$U(v_1)\topprec U(\gamma).$$
By Lemma \ref{top order and geodesic},
$\pi_{k_W}(v_1) \le \pi_{k_W}(\gamma)+r$, which implies that
$$d_W(\gamma,\tau_W)\le a_1+r.$$

We have uniformly bounded $d_W(\gamma,\tau_W)$ in all cases. In
combination with (\ref{reduction to t_W}), this completes the proof of the second claim
in our simplified case.

\subsubsection*{General ordering argument}
We now adapt the above argument to hold in the general setting where (S1) and (S2) may not hold.
It will be convenient to divide $\partial W$ into  the collection  $\partial_1W$ of components which do
not overlap $\alpha$, and the collection $\partial_0W$ of components which do overlap $\alpha$. Moreover,
we will assume that the pleated surface $f$ realizes $\alpha\cup\partial_1W$.  
As in the simplified case,
we will construct extended split level surfaces $G_0$, $G_1$ and $G_2$ such that $G_0$ and $G_2$ 
are the horizontal boundary components of a $W$-product region $R$ which
contains $G_1$. We will choose the spacing constant  to guarantee that
\begin{enumerate}
\item
$f(X)$ is disjoint from $G_1$,
\item
there exists a collection of annuli $A_0$ joining (suitable components
of) $\partial G_1$ to $(\partial_0W)^*$ such that $f(X)$ misses $A_0$, and
\item
there exists a homotopy from $f(\gamma)$ into $h(U(\gamma))$ which is disjoint from $G_1$.
\end{enumerate}
By (1) and (2), $f(X)$ is disjoint from $\bar G_1=G_1\cap A_0$ and therefore, by Lemma \ref{above or below}, $f(X)$ lies above or below $\bar G_1$ and, as in the simplified case, Lemma \ref{above and below} implies that $\alpha^*$ lies either above or below every component of $(\partial W)^*$ that it overlaps. This proves the ordering statement. 
Since $f(X)$ lies above $\bar G_1$ and hence above $G_1$, then (3) implies
that $h(U(\gamma))$ lies above $G_1$ if it is disjoint from $G_1$. Therefore, once we have established (1)--(3), we can 
use Lemma \ref{top order and geodesic} to complete the proof just as we did in the simplified case.

We remark that with a little more work, we could homotope $f$ to a uniformly Lipschitz map $g$ such that
$g(\gamma)\subset h(U(\gamma))$ and  $g(X)$ is disjoint from $\MT(\partial_0 W)$. Our proof would then
resemble the simplified case even more closely.

We now establish (1)--(3). Notice that our assumptions imply that 
$$f(\collar(\partial_1W))\subset\MT(\partial_1 W).$$
Since $l_X(\alpha)\le L$ and $\alpha$ overlaps every component of $\partial_0W$, there
exists $L_2=L_2(S,L)$ so that passing through every point in $X-\collar(\partial_1(W))$ there is
a curve of length at most $L_2$ which is not homotopic into
$\collar(\partial W)$.
Moreover, there
exists $\ep_3=\ep_3(\max\{L_2,B,K_h\})$ so that any curve of length at most
$\max\{L_2,B,K_h\}$ which intersects $\MT_{\ep_3}(\partial W)$ is
contained in $\MT(\partial W)$.
It follows that $f(X-\collar(\partial_1W))$ is
disjoint from $\MT_{\ep_3}(\partial W)$.

Applying Theorem \ref{kgccfact}, 
we may assume that $d_W(\nu^+,\nu^-)$ is large enough that 
$l_\rho(\partial W)\le \ep_3/4$. Therefore, there exists a constant $K_3$ so that the radial projection
$$r_W:\MT(\partial W)-\MT_{\ep_3}(\partial W)\to \partial \MT(\partial W)$$
is $K_3$-Lipschitz. We may extend $r_W$ to a $K_3$-Lipschitz map defined on  \hbox{$N-\MT_{\ep_3}(\partial W)$}
by setting it to be the identity off of $\MT(\partial W)$.

Let $Z$ be a component of $X-\collar(\partial_1 W)$. Lemma \ref{split BDL} implies that there exists $b=b(S,\ep_3)$ so
that 
$$\thdiam_{X_{thick(\ep_3)}}(Z)\le b.$$
Since $f(\gamma)$ has length at most $B$ and the core curve of $h(U(\gamma))$ has length at most $K_h$,
one may homotope $f(\gamma)$ into
$h(U(\gamma))$  through a family of curves of length at most $\max\{B,K_h\}$. 
It follows that this homotopy may be completed in the complement of $\MT_{\ep_3}(\partial W)$.
Therefore, there exists a constant $b_1=b_1(B,\ep_3)$ so that there is a homotopy from $f(\gamma)$
into $h(U(\gamma))$ all of whose tracks have length at most $b_1$ and are disjoint from $\MT_{\ep_3}(\partial W)$.
(If $\ell_\rho(\gamma)<\ep_h$, then $h(U(\gamma))=\MT(\gamma)$ and this follows from Lemma 2.6 in
\cite{ELC2}. If not, then the ruled homotopies from $f(\gamma)$ to $\gamma^*$ and from the core curve
of $h(U(\gamma))$  to $\gamma^*$ each have uniformly bounded length tracks  and avoid $\MT_{\ep_3}(\partial W)$, so
they may be concatenated
to produce the desired homotopy.)

Let $c_3=c_3(S,\ep_3)$ be the constant given by Lemma \ref{W cross sections}, and 
assume that
$$a_1 \ge \lceil K_3(b+b_1)/c_3 \rceil.$$
As in the simplified case, this implies that $k_W$ has length at least $2a_1$.

Let $v_0<v_1<v_2=\tau_W$ be three vertices in $k_W$ such that \hbox{$d_W(v_i,v_{i+1}) = a_1$.}
Lemma \ref{W cross  sections} provides extended split level surfaces
$G_i = h(\hhat F_{v_i})$ and a $W$-product region $R$ with
horizontal boundary $G_0\cup G_2$ such that $G_1\subset R$
and
\begin{equation}\label{F1 Fj thick far simpG}
\thd_{N_{thick(\ep_3)},R}(G_1,G_j) >K_3 (b+b_1)
\end{equation}
for $j=0$ and $j=2$.

We first claim that  if $Z$ is a component of $X-\collar(\partial_1W)$, then $f(Z)$ is disjoint from $G_1$. 
Since $f(\collar(\partial_1W))\subset \MT(\partial_1(W))$, 
this implies that $f(X)$ is disjoint from $G_1$, which establishes (1). 
Notice that $Z$ is not homotopic into $W$, so $f(Z)$ cannot be contained in
$R\cup\MT(\partial W)$. So, if $f(Z)$ intersects $G_1$, 
there is a path $\nu$ in $Z$ so that $\nu\cap Z_{thick(\ep_3)}$ has length at most $b$ and
$f(\nu)$ joins $G_1$ to a point outside of $R\cup\MT(\partial W)$. 
Since $f( Z_{thin(\ep_3)})$ is disjoint from $\MT_{\ep_3}(\partial W)$ and $f(Z_{thin(\ep_3)})\subset N_{thin(\ep_3)}$,
$\bar\nu=r_W(f(\nu))$ is a path contained in $N-\MT(\partial W)$ joining
a point on $G_1$ to a point outside of $R$ such that $\bar\nu\cap N_{thin(\ep_3)}$ has length at most $K_3b$.
However this would contradict
(\ref{F1 Fj thick far simpG}).

A very similar argument establishes (2). 
Let $A_0$ be the collection of radial annuli in $\MT(\partial_0W)$ joining components of $(\partial_0W)^*$ to
the appropriate components of $\partial G_1$.
If $Z$ is a component of $X-\collar(\partial_1W)$ and
$f(Z)$ intersects $A_0$, then there is a path $\nu$ in $Z$ so that
$\nu\cap Z_{thick(\ep_3)}$ has length at most $b$ and $f(\nu)$ joins $A_0$ to a point outside of $R\cup\MT(\partial_0W)$. 
Then $\bar\nu=r_W(f(\nu))$ would be a path contained in $N-\MT(\partial W)$ joining
a point on $G_1$ to a point outside of $R$ such that $\bar\nu\cap N_{thin(\ep_3)}$ has length at most $K_3b$.
Again, this contradicts (\ref{F1 Fj thick far simpG}). Therefore, $f(X)$ is disjoint from $A_0$ and we have established (2).

To establish (3), consider the homotopy with track lengths at most $b_1$ joining $f(\gamma)$ to a curve
in $h(U(\gamma))$ (in the complement of $\MT_{\ep_3}(\partial W)$). If  the homotopy
intersects $G_1$ then there is a path $\eta$ of length $b_1$ joining $G_1$ to
$f(X)$. Then $r_W(\eta)$  is a path of length $Kb_1$ in $N-\MT(\partial W)$ joining $G_1$ to a point in
$z\in f(X)$. One may then apply the construction in (1) to find a path $\bar \nu$ in $N-\MT(\partial W)$ joining
$r_W(z)$ to a point in $N-R$
such that  $\bar\nu\cap N_{thin(\ep_3)}$ has length at most $K_3b$. Concatenating $\eta$ and $\bar \nu$ would
again contradict (\ref{F1 Fj thick far simpG}). This completes the proof of the first claim in the non-annular case.
\end{proof}

\subsection{Proof in the annular case}
\label{overlap annular}

We now proceed to give a proof of Theorem \ref{topistop} in the case when $W$ is an annulus.
Let $\gamma$ be the core of $W$. 
Assume that $\alpha\in C(\rho,L)$ and that $\alpha^*$ lies above $\gamma^*$.

We first observe that we may assume that there is a bounded length curve $\beta$ such that
$d_W(\beta,\nu^+)\le1$ whose geodesic representative lies above $\gamma^*$.

\begin{lemma}{nearby curve}
Let  $S$ be a compact surface and $L\ge L_0$. There exists $K_4=K_4(S,L)$  and $L_1=L_1(S,L)$
such that if
$\rho\in AH(S)$ has end invariants $\nu^\pm$,
$\alpha\in \CC(\rho,L)$, $\gamma\in \CC(S)$, $\alpha^*$ lies above $\gamma^*$ and
$$d_\gamma(\alpha,\nu^+)>K_3,$$
then there exists $\beta\in\CC(\rho,L_1)$ such that $d_\gamma(\beta,\nu^+)\le 1$
and $\beta^*$ lies above $\gamma^*$.
\end{lemma}

\begin{proof}
First notice that
if $\gamma$ overlaps a simple closed curve component of $\nu^+$ then we can choose
$\beta$ to be that curve. (Recall that, by convention, we say that $\beta^*$ lies above $\gamma^*$ if
$\beta$ is an upward-pointing parabolic).

If $\gamma$ overlaps a lamination component $\lambda$ of $\nu^+$ with support $Y$,
then there exists a sequence $\{\beta_n\}\subset\CC(\rho,L_h)\cap \CC(Y)$ so that
$\beta_n\to\lambda$ in $\CC(Y)$. For all sufficiently large $n$,
$d_W(\beta_n,\nu^+)\le 1$ and $\beta_n^*$ lies above $\gamma^*$, so we may choose
$\beta=\beta_n$ for some specific large enough $n$.

In the remaining case, $\gamma$ is contained in $Y$ where $Y$ is the support of an upward
pointing geometrically finite end.  Let $Y_h$ be the component of the
(upper) convex hull boundary 
associated to $Y$, and let $Y_\infty$ be the
corresponding component of the boundary at infinity, with its
Poincar\'e metric. Recall that there is a 2-Lipschitz map
$r:Y_\infty\to Y_h$ in the correct homotopy class, called the nearest
point retraction \cite[Theorem 3.1]{EMM}.

There exists $\epsilon_4=\epsilon_4(L)$ such that any curve of length at most
$L$ which intersects 
$\MT_{\ep_4}(\gamma)$ is homotopic to a power of $\gamma$. If
$l_{Y_h}(\gamma)<\ep_4$,
then there is an annulus $A$ in $\MT_{\ep_4}(\gamma)$ joining
$\gamma^*$ to its representative on $Y_h$. But then $\alpha^*$, which
lies in $C(N)$ and  above $\gamma^*$, would be forced to intersect
$A$, contradicting the assumption that $l(\alpha^*)<L$. 
Therefore, we may assume that $l_{Y_h}(\gamma)\ge \ep_4$, which
implies (via the map $r$) that $l_{Y_\infty}(\gamma) \ge \ep_4/2$.
It follows that the minimal $Y_\infty$-length marking $\nu^+|_Y$ contains a curve $\beta$ of length at most
$L_1=L_1(S,\ep_4)$ which overlaps $\gamma$.

It is clear that $r(\beta)$ lies above $\gamma^*$, unless it intersects $\gamma^*$, since $\gamma^*\subset C(N)$.
There exists $\ep_5=\ep_5(L_1)$ so that there is a homotopy from $r(\beta)$ to $\beta^*$ which is
disjoint from $\MT_{\ep_5}(\delta)$ for any curve $\delta\in\CC(S)-\{\beta\}$. Therefore, if
$l_\rho(\gamma)<\ep_5$, then $\beta^*$ lies above $\gamma^*$.
Theorem \ref{kgccfact} gives 
$K_4=K(S,L,\ep_5)$ so that if
$\diam_\gamma(\CC(\rho,L))\ge K_4$, then $l_\rho(\gamma)<\ep_5$.
But, since
$$\diam_\gamma((C(\rho,L))\ge d_\gamma(\alpha,\nu^+)\ge K_4$$
by assumption, we
may conclude that $\beta^*$ lies above $\gamma^*$ in this case as well.
\end{proof}

The annular case of Theorem \ref{topistop} then follows quickly from the following result:

\begin{proposition}{annulus case}
Given a compact surface $S$ and $L,D>0$ ,there exists $F=F(D,L)$ such that if $\rho \in AH(S)$ and
$\alpha$, $\beta$ and $\gamma$ are curves in  $\CC(S)$  such that
\begin{enumerate}
\item $\alpha$ and $\beta$ overlap $\gamma$,
\item $\ell_\rho(\alpha)\le L$ and $\ell_\rho(\beta)\le L$,
\item $\alpha^*$ and $\beta^*$ lie above $\gamma$, and
\item if $Y$ is non-annular essential subsurface with $\gamma\in [\partial Y]$,
then 
$$d_Y(\alpha,\beta) \le D.$$
\end{enumerate}
then $$d_\gamma(\alpha,\beta) \le F.$$
\end{proposition}

\begin{proof}[Proof of the annular case of Theorem \ref{topistop}]
We may assume that \hbox{$d_W(\alpha,\nu^+)>K_4$} where $K_4=K_3(S,L)$ is the constant from Lemma
\ref{nearby curve}. Let $\beta$ be the curve provided by Lemma \ref{nearby curve}. Recall that
$l_\rho(\beta)<L_1=L_1(S,L)$, $d_\gamma(\beta,\nu^+)\le 1$, and $\beta^*$ lies above $\gamma^*$. 

If $D=D(S,L)$ is the constant from the non-annular case of Theorem \ref{topistop}, then
if $Y$ is any non-annular surface with $\gamma\in [\partial Y]$, then 
$$d_Y(\alpha,\beta)\le 2D.$$
We may then apply Proposition \ref{annulus case} to conclude that
$$d_\gamma(\alpha,\beta) \le F=F(2D,\max\{L,L_1\}).$$
It follows that
$$d_\gamma(\alpha,\nu^+)\le F+1$$
and the proof is complete.
\end{proof}

\medskip

We now turn to the proof of Proposition \ref{annulus case}.
The first lemma we need is a mild variation of
the model manifold theorem from \cite{ELC1}, in
which the end invariants have been replaced by the bounded-length
curves $\alpha$ and $\beta$. Our statement of this lemma
forgets most of the structure of the model and the Lipschitz
properties of the map, remembering only the properties of the map concerning
the tube $U(\gamma)$ and the images of  $\alpha$ and $\beta$. 
We will have to take a bit of care because the statements in 
\cite{ELC1} are given in the setting where the initial and terminal
markings of the hierarchy are actually the end invariants of
$\rho$, although the proofs go through verbatim in this setting. 
In Bowditch \cite{bowditch}, simpler proofs of the a priori length
bounds  of \cite{ELC1} are given, in the more general setting, and
this will simplify the discussion. We will point out the details as we go.

\begin{lemma}{special model}
Given a compact surface $S$ and $L, D>0$ there exist $K_1=K_1(S,L,D)$ and $K_2=K_2(S,L,B)$ such that if
$\rho\in AH(S)$,
$\alpha$ and $\beta$ are intersecting curves in $\CC(S)$
and  \hbox{$\gamma\in\CC(S)$} intersects both $\alpha$ and $\beta$,
\begin{itemize}
\item $\ell_\rho(\alpha),\ell_\rho(\beta) \le L$,

\item $d_\gamma(\alpha,\beta) > K_1$, and

\item $d_Y(\alpha,\beta) < D$ if $\gamma \in [\boundary Y]$ and
$Y$ is non-annular,
\end{itemize}
then there exists a map of pairs
$$f_{\alpha,\beta}: (S\times[0,1], \boundary S\times[0,1])  \to
(N^0_\rho,\boundary N^0_\rho),$$
in the homotopy class determined by $\rho$, such that
\begin{enumerate}
\item The preimage $f_{\alpha,\beta}^{-1}(\MT(\gamma))$ is a
  regular neighborhood $U(\gamma)$ of $\gamma\times\{1/2\}$.
\item The restriction of $f_{\alpha,\beta}$ to $\boundary U(\gamma)$,
  as a map to $\boundary \MT_\rho(\gamma)$, 
  is $K_2$-Lipschitz with respect to a Euclidean metric $\sigma$ on
  $\boundary U(\gamma)$ which has area at most $K_2$, and meridian
  length bounded below by $d_\gamma(\alpha,\beta)/K_2$.
\item If $\alpha^*$ lies above (respectively below) $\gamma^*$,
then $f|_{S\times\{0\}}$ lies above (respectively below) $\gamma^*$.
\item If $\beta^*$ lies above (respectively below) $\gamma^*$,
then $f|_{S\times\{1\}}$ lies above (respectively below) $\gamma^*$.
\end{enumerate}

\end{lemma}

\begin{proof} We first extend $\alpha$ to a  pants decomposition $P_\alpha$ 
each of whose curves has length at most $L'=L'(L)$. We do so by considering
a pleated surface \hbox{$f:X\to N$} realizing $\alpha$ and letting $P_\alpha$ be a minimal
length pants decomposition of $X$ which contains $\alpha$.
We similarly extend $\beta$ to a pants decomposition $P_\beta$ 
each of whose curves has length at most $L'=L'(L)$.

Construct a hierarchy $H$ whose initial and terminal markings are
$P_\alpha$ and $P_\beta$, respectively. Section 7 of \cite{ELC1}, as
well as the main theorem of Bowditch \cite{bowditch}, give us an uniform upper
bound on the lengths of all curves  in $\CH$. 

Choose $K_1>K(S,\ep_h,L)$, where $K(S,\ep_h,L)$ is the constant from Theorem \ref{kgccfact}, 
so that, with our assumptions, $\ell_\rho(\gamma) < \ep_h$. Moreover, we may choose
$K_1>A(S)$ where $A(S)$ is the constant from Lemma \ref{large link}, so that
$\gamma$ must lie in $\CH$.

We assume, for the moment, that $P_\alpha$ and $P_\beta$ have no
curves in common. The construction in Section 8 of \cite{ELC1} produces 
a ``model manifold'' from the hierarchy $H$. This is a manifold,
equipped with a path metric, homeomorphic to $S\times[0,1]$ minus
the curves $P_\alpha\times\{0\}$ and $P_\beta\times\{1\}$. To each
curve $v$ in $H$ is associated a ``tube'' $U(v)$ which is a solid
torus regular neighborhood of a level curve isotopic to $v$. 
This applies in particular to $\gamma$ and the curves of $P_\alpha$
and $P_\beta$. After removing $\bbar U(v)$ for each curve
$v$ in $P_\alpha$ and $P_\beta$, and taking the closure of what
remains (this has the effect of removing the part of $\boundary U(v)$
that is in the boundary of the model), we obtain a subset $M$ of the model that is homeomorphic to
$S\times[0,1]$, and such that $P_\alpha$ and $P_\beta$ are realized with
uniformly bounded length on its boundary (we again label these curves
$P_\alpha\times\{0\}$ and $P_\beta\times\{1\}$). This is the manifold on
which we will define our map.

The boundary of the tube $U(\gamma)$ is a Euclidean torus, whose geometry
is controlled in terms of the coefficients
$\{d_Y(P_\alpha,P_\beta)\}_{\gamma\subset\boundary Y}$, by Theorem 9.1
(and the discussion in Section 9) in \cite{ELC1}. 
In particular, given our uniform  upper bound on $d_Y(P_\alpha,P_\beta)$ over all
non-annular $Y$ with $\gamma \subset \boundary Y$, we obtain an  uniform upper
bound on the area of $\boundary U(\gamma)$. 
The same theorem gives a lower bound of the form $d_\gamma(\alpha,\beta)/K_2$ for
the length of the meridian of this torus. Together these bounds give us conclusion (2). 

We now note that the construction of a map
from $M$ to $N_\rho$ carried out in Section 10 of \cite{ELC1},
depends only on the length bounds on hierarchy curves. Thus the proof
of the Lipschitz Model Theorem in that section, carried out in our
setting, yields 
a continuous map $f:M\to N_\rho$ such that $f^{-1}(\MT(\gamma)) = U(\gamma)$
(conclusion (3) of that theorem),
takes $\alpha\times\{0\}$ and $\beta\times\{1\}$ to curves of
uniformly bounded length, and is $K_2$-Lipschitz  on $\partial \MT(\gamma)$
(for some uniform choice of $K_2$ depending only on $S$ and $L$).

Since $f(\alpha\times\{0\})$ has bounded length, there exists $\ep_6>0$ so that the ruled homotopy
from $f(\alpha\times \{0\})$ to $\alpha^*$ cannot intersect $\MT_{\ep_6}(\alpha)$.
We again use Theorem \ref{kgccfact} to observe that we may choose
$K_1$ large enough that 
\hbox{$l_\rho(\gamma)<\ep_6/2$}, so that the ruled homotopy is disjoint from $\gamma^*$.
Therefore, 
if $\alpha^*$ lies above (below) $\gamma^*$, then
$f(\alpha\times\{0\})$ lies above (below) $\gamma^*$, which implies, by Lemma \ref{above and below},
that $f|_{S\times \{0\}}$ lies above (below) $\gamma^*$. 
Similarly, if $\beta^*$ lies above (below) $\gamma^*$, then
$f(\beta\times\{1\})$ lies above (below) $\gamma^*$, which implies, by Lemma \ref{above and below},
that $f|_{S\times \{1\}}$ lies above (below) $\gamma^*$. 

Thus we obtain a map satisfying all the conditions of the lemma. This
completes the proof under our assumption that $P_\alpha$ and $P_\beta$
have no common curves.

\medskip

We now explain how to handle the case when $P_\alpha$ and $P_\beta$ do have common curves. Let 
$\Delta$ denote the union of their shared components.

A hierarchy between $P_\alpha$ and $P_\beta$ still exists, and the 
a priori length bounds on the hierarchy curves
are obtained in exactly the same way. The construction of the model,
however, is now slightly different: In \cite{ELC1}, the initial and
terminal markings are allowed to have common components only if at
least one of them comes with a transversal. This is not the case here,
so we must use a variation of the construction.

The construction in \cite{ELC1} proceeds just as before, except at the
stage where a tube is inserted for a component $\delta$ of
$\Delta$. This tube will now be of the form $(annulus)\times\R$, and
so the removal of the tubes associated to $\Delta$ will produce a
model naturally homeomorphic to $R\times [0,1]$, where $R$ is the
(possibly disconnected) surface $S\setminus \collar(\Delta)$.

The construction of the Lipschitz map $f$ proceeds as before on each
component of $R\times[0,1]$. We note that the restriction of the model
map to $R\times\{1\}$ may be extend to a map defined on $S\times\{1\}$
by appending ruled annuli connecting $f(\partial R)$ to $\partial
R^*$.  Since there is an uniform upper bound on the the length of
$f(\partial R)$, there exists $\ep_7>0$, so that these appended annuli
cannot intersect $\MT_{\ep_7}(\gamma)$.  We may assume that $K_1$ is
chosen large enough that \hbox{$l_\rho(\gamma)<\ep_7/2$}, so that a
retraction can be used to adjust the map so that the pre-image of
$\MT(\gamma)$ is just $U(\gamma)$.  We can extend this to a slight
thickening of $S\times\{1\}$, so that our final model is homeomorphic
to $S\times[0,1]$, and again the proof is complete.

\end{proof}

We next establish a lower bound on the Lipschitz constant of a degree zero map between two Euclidean tori,
which depends on the minimal length of a homotopically non-trivial geodesic generating the kernel of
the associated map between the fundamental groups.

\begin{lemma}{lipschitz_torus}
If $T_0$ and $T_1$ are Euclidean tori and $f:T_0 \rightarrow T_1$ is  a Lipschitz  map  such that
the kernel of $f_*:\pi_1(T_0)\to \pi_1(T_1)$ is infinite cyclic and is generated by an element whose 
geodesic representative has length $L$, then
$${\rm Lip}(f)\ge\frac{2\inj(T_1)L}{\Area(T_0)},$$
where ${\rm Lip}(f)$ is the minimal Lipschitz constant for $f$.
\end{lemma}

\begin{proof}
Let $a$ be a generator of $\ker( f_*)$. Then $a$ stabilizes a line $\ell_0$ in 
the universal cover $\tilde{T}_0 = \R^2$. Choose $b \in \pi_1(T_0)$ so that $a$ and $b$ generate
and let $\ell_n = b^n(\ell_0)$. Notice that $\Area(T_0) = L d(\ell_0, \ell_1)$. 
It follows that
$$d(\ell_0, \ell_n) = \frac{n\Area(T_0)}{L}.$$

The lift $\tilde{f}: \tilde{T}_0 \rightarrow \tilde{T}_1$ factors through a map to $\tilde{T}_0/\langle a \rangle$ 
and in $\tilde{T}_0/\langle a \rangle$ the image of $\ell_0$ is compact. Notice that
$\diam(\tilde{f}(\ell_0))=\diam(\tilde{f}(\ell_n)).$
Since $f_*(b)$  act as translation on $\tilde{T}_1 = \R^2$ with translation distance at least $ 2\inj(T_1)$,
we see that
$$d(\tilde{f}(\ell_0), \tilde{f}(\ell_n)) \ge 2\inj(T_1)n - 2\diam(\tilde{f}(\ell_0)).$$
Therefore, 
$$\operatorname{Lip}(f) \ge \operatorname{Lip}(\tilde{f}) \ge \frac{\left(2\inj(T_1)n - 2\diam(\tilde{f}(\ell_0))\right)L}{n\Area(T_0)}.$$
Letting $n \rightarrow \infty$ gives the desired estimate.
\end{proof}

\begin{proof}[Proof of Proposition \ref{annulus case}]
We may assume that $d_\gamma(\alpha,\beta)>K_1$,  where $K_1$ is the constant from Lemma \ref{special model},
since otherwise we are done.

Let $f_{\alpha,\beta}$ be the map given by Lemma \ref{special model}
and let  $f: \boundary U(\gamma) \to \boundary\MT(\gamma)$ be the restriction of
$f_{\alpha, \beta}$ to the torus $\boundary U(\gamma)$. Since $f_{\alpha,\beta}|_{S\times \{0\}}$ and
$f_{\alpha,\beta}|_{S\times \{1\}}$ both lie above $\gamma^*$ and have image disjoint from $\MT(\gamma)$, they are homotopic
in the complement of $\MT(\gamma)$.
Lemma \ref{both above} then implies that  $\deg(f) = 0$ and thus  that $\ker(f) \neq \{\id\}$. 
Since $f_{\alpha, \beta}$ is a homotopy equivalence the kernel of $f$ has to be contained in the kernel of the inclusion of 
$\boundary U(\gamma)$ in $S \times [0,1]$. The kernel of this second map is generated by the meridian of $\boundary U(\gamma)$ and therefore $\ker(f)$ is generated by a power of the meridian.

By Lemma \ref{special model} there exists a constant $K_2>0$ such that  $\boundary U(\gamma)$ is an Euclidean torus with area bounded above by $K_2$, the length of the meridian is bounded below by $d_\gamma(\alpha,\beta)/K_2$ and $f$ is a 
$K_2$-Lipschitz map. The boundary of the Margulis tube $\MT(\gamma)$ is also a Euclidean torus and its injectivity radius
is uniformly bounded below by some constant $C_1$.
We then apply
Lemma \ref{lipschitz_torus} to conclude that
$$K_2> \frac{2C_1(d_\gamma(\alpha,\beta)/K_2)}{K_2} =\frac{ 2C_1d_\gamma(\alpha,\beta)}{(K_2)^2}.$$
Rearranging we have 
$$d_\gamma(\alpha,\beta)< \frac{(K_2)^3}{2C_1}$$
which completes the proof.
\end{proof}

\end{document}